\documentclass[a4paper,11pt]{amsart}

\usepackage[all]{xy}

\usepackage[utf8]{inputenc}						
\usepackage[T1]{fontenc}
\usepackage[english]{babel}

\usepackage{amsfonts}							
\usepackage{amsmath}
\usepackage{amssymb}
\usepackage{amsthm}
\usepackage{mathtools}

\usepackage{graphicx}							
\usepackage{tikz}
\usepackage{tikzpagenodes}

\usepackage{hyperref}							
	\hypersetup{colorlinks=false, urlcolor=black, linkcolor=black}

\usepackage{enumitem}							

\newcommand{\N}{\mathbb{N}}						
\newcommand{\Z}{\mathbb{Z}}						
\newcommand{\R}{\mathbb{R}}						
\newcommand{\C}{\mathbb{C}}						

\newcommand{\bS}{\mathbb{S}}					
\newcommand{\bT}{\mathbb{T}}					

\newcommand{\eps}{\varepsilon}					

\newcommand{\dd}								
	{\mathop{}\!\mathrm{d}}						

\DeclareMathOperator{\id}{id}					
\DeclareMathOperator{\vol}{vol}					
\DeclareMathOperator{\tor}{Tor}					
\DeclareMathOperator{\spt}{spt}					
\DeclareMathOperator{\mass}{\mathbf{M}}

\DeclareMathOperator{\im}{im}					

\newcommand{\abs}[1]{\left| #1 \right|}			
\newcommand{\norm}[1]							
	{\lVert #1 \rVert}

\newcommand{\push}[1]{{#1}_*\,}					
\newcommand{\hodge}{\mathtt{\star}\hspace{1pt}}	

\newcommand{\free}{\mathrm{free}}				
\newcommand{\loc}{\mathrm{loc}}
\newcommand{\CE}{\mathrm{CE}}

\newcommand{\cesob}{W^{\text{OCE}}}				
\newcommand{\cesobt}{W^{\text{CE}}}				
\newcommand{\cesobloc}
	{W^{\text{OCE}}_{\text{loc}}}
\newcommand{\cesobtloc}
	{W^{\text{CE}}_{\text{loc}}}
\newcommand{\cehom}[1]{H_{\text{CE}}^{#1}}

\newcommand{\sheaf}{\mathcal{S}}				
\newcommand{\Rsheaf}{\mathcal{R}}				
\newcommand{\cesheaf}[2]
	{\mathcal{S}^{#1}_{#2, \text{CE}}}
\newcommand{\singsheaf}[2]
	{\mathcal{S}^{#1}_{#2, \text{sing}}}
\newcommand{\Fetale}{\mathcal{F}}
\newcommand{\ceetale}[2]
	{\mathcal{F}^{#1}_{#2, \text{CE}}}

\newtheorem{thm}{Theorem}[section]{\bf}{\it}
\newtheorem{lemma}[thm]{Lemma}
\newtheorem{prop}[thm]{Proposition}
\newtheorem{cor}[thm]{Corollary}

\theoremstyle{definition}
\newtheorem{defn}[thm]{Definition}

\theoremstyle{remark}
\newtheorem{rem}[thm]{Remark}

\numberwithin{equation}{section}

\begin{document}

\title[Uniform cohomological expansion]{Uniform cohomological expansion of uniformly quasiregular mappings}
\author{Ilmari Kangasniemi \and Pekka Pankka}
\address{Department of Mathematics and Statistics, P.O. Box 68 (Gustaf H\"allstr\"omin katu 2b), FI-00014 University of Helsinki, Finland}
\email{\{ilmari.kangasniemi,pekka.pankka\}@helsinki.fi}

\begin{abstract}
Let $f\colon M \to M$ be a uniformly quasiregular self-map of a compact, connected, and oriented Riemannian $n$-manifold $M$ without boundary, $n\ge 2$. We show that, for $k \in \{0,\ldots, n\}$, the induced homomorphism $f^* \colon H^k(M;\R) \to H^k(M;\R)$, where $H^k(M;\R)$ is the $k$-th singular cohomology of $M$, is complex diagonalizable and the eigenvalues of $f^*$ have absolute value $(\deg f)^{k/n}$. As an application, we obtain a degree restriction for uniformly quasiregular self-maps of closed manifolds. In the proof of the main theorem, we use a Sobolev--de Rham cohomology based on conformally invariant differential forms and an induced push-forward operator.
\end{abstract}

\thanks{This work was supported in part by the doctoral program DOMAST of the University of Helsinki and the Academy of Finland project \#297258. This is the final draft version accepted for publication. For the published version, see: \emph{Proc.\ London Math.\ Soc.}, 118(3):701--728, 2019. (\url{https://doi.org/10.1112/plms.12205}).}
\subjclass[2010]{Primary 30C65; Secondary 57M12, 30D05}
\keywords{Uniformly quasiregular mappings, degree spectrum, entropy, Sobolev--de~Rham cohomology}
\date{\today}

\maketitle

\section{Introduction}

A continuous self-map $f\colon M \to M$ of an oriented Riemannian $n$-manifold $M$ of dimension $n \geq 2$ is \emph{$K$-quasiregular for $K\ge 1$} if $f$ belongs to the Sobolev space $W^{1,n}_\loc(M,M)$ and satisfies the inequality 
\[
	\abs{Df}^n \leq K J_f \quad \text{a.e. in}\ M,
\]
where $\abs{Df}$ is the operator norm of the differential $Df$ of $f$ and $J_f$ the Jacobian determinant $\det Df$. A quasiregular self-map $f\colon M\to M$ is \emph{uniformly $K$-quasiregular} if all iterates $f^k$ of $f$ for $k\ge 1$ are $K$-quasiregular. By a result of Iwaniec and Martin \cite{Iwaniec-Martin_AASF}, uniformly quasiregular mappings preserve a bounded measurable conformal structure, and hence are also termed \emph{rational quasiregular maps}. We refer to Martin \cite{Martin2014} for an extensive survey.  

In this article, we show that uniformly quasiregular self-maps of degree at least two on a closed manifold are uniformly cohomologically expanding. In what follows, $H^*(M;\R)$ denotes the \emph{singular cohomology of the manifold $M$ with real coefficients}. Recall that a manifold is \emph{closed} if it is compact and without boundary. Our main theorem reads as follows.
	
\begin{thm}\label{alleigenvaluessamediag}
	Let $f \colon M \to M$ a uniformly quasiregular map on a closed, connected, oriented Riemannian $n$-manifold $M$, $n\ge 2$, and let $k \in \{0, \ldots, n\}$. Then the induced map $f^* \colon H^k(M;\R) \to H^k(M;\R)$ is complex diagonalizable and all (complex) eigenvalues of $f^*$ have modulus $(\deg f)^{k/n}$, where $\deg f$ denotes the topological degree of $f$.
\end{thm}

As an immediate consequence of Theorem \ref{alleigenvaluessamediag} we obtain a cohomological obstruction for the uniformly quasiregular Sto\"ilow theorem of Martin and Peltonen \cite{Martin-Peltonen-PAMS}. To be more precise, recall that, by a theorem of Martin and Peltonen, given a quasiregular self-map $F\colon \bS^n \to \bS^n$ of the $n$-sphere $\bS^n$ for $n\ge 2$, there exists a quasiconformal homeomorphism $h \colon \bS^n \to \bS^n$ for which the composition $f = F\circ h$ is uniformly quasiregular. Having Theorem \ref{alleigenvaluessamediag} at our disposal, it is easy to find quasiregular self-maps of closed manifolds for which the homology gives an obstruction for this factorization property. Consider, for example, for $m\ge 2$ the stretch map $f \colon re^{i \theta} \mapsto r e^{i m \theta}$ on $\bS^1$, a winding map $F \colon (re^{i\theta}, t) \mapsto (re^{im\theta},t)$ on $\bS^2 \subset \C\times \R$, and their product $f\times F \colon \bS^1\times \bS^2\to \bS^1\times \bS^2$. Then $f\times F$ is quasiregular, but there are no self-homeomorphisms $\psi$ and $\varphi$ of $\bS^1\times \bS^2$ for which the composition $\tilde f=\psi \circ (f\times F)\circ \varphi$ is uniformly quasiregular, since the homomorphism $\tilde f^* \colon H^2(\bS^1\times \bS^2) \to H^2(\bS^1\times \bS^2)$ induced by this composition is $c\mapsto mc$ and $m\ne m^{4/3} = (\deg f)^{2/3}$. 

Theorem \ref{alleigenvaluessamediag} is sharp in two ways. First, simple examples show that the second claim on the moduli of the eigenvalues does not hold if $M$ is an open manifold. For example, let $f\colon \C^* \to \C^*$ be the standard power map $z\mapsto z^m$, for $m\ge 2$, on the punctured complex plane $\C^* = \C\setminus \{0\}$. Then $f^* \colon H^1(\C^*;\R) \to H^1(\C^*;\R)$ is the homomorphism $x \mapsto m x$, and $m=\deg f \ne (\deg f)^{1/2}$. We find it interesting that, in this case, the induced homomorphism $f^* \colon H^1_c(\C^*;\R) \to H^1_c(\C^*;\R)$ in the compactly supported cohomology is the identity and, in particular, is independent on the degree. This follows, for example, from Poincar\'e duality. 

Second, we observe that the eigenvalues of $f^*$ need not be real, even after iteration. Consider, for example, the linear mapping $L \colon \R^2 \to \R^2$, $(x, y) \mapsto (3x - 4y, 4x + 3y)$. Then $L$ induces a Latt\`es map $f \colon \bT^2 \to \bT^2$ on the $2$-torus $\bT^2$ and the matrix of $f^* \colon H^1(\bT^2;\R)\to H^1(\bT^2;\R)$ with respect to the standard basis of $H^1(\bT^2;\R)$ is just the matrix of $L$ in the standard basis of $\R^2$. Hence, the eigenvalues $3 \pm 4i$ of $f^*$ are not real. Since $(3+4i)/5$ is not a root of unity, the eigenvalues of $(f^m)^*$ are not real for any $m\in \Z_+$. Indeed, the minimal rational polynomial $P$ of $(3+4i)/5$ is $z \mapsto z^2 - (6/5)z + 1$. Since $P$ has a non-integer coefficient, $P$ is not cyclotomic, and consequently $(3+4i)/5$ is not a root of unity. Thus $(3+4i)^m$ is not real for any $m \in \Z_+$.

\subsection*{Uniformly quasiregular mappings and entropy}

Our interest in the eigenvalues of $f^* \colon H^*(M;\R)\to H^*(M;\R)$ is influenced in part by questions related to the topological entropy of uniformly quasiregular maps; see e.g.\ Gromov \cite{Gromov-2003}. In particular, Shub's entropy conjecture for $C^1$-mappings $f\colon M\to M$, solved by Yomdin in the $C^\infty$-case, states that 
\begin{equation}\label{eq:Shub}
	h(f) \ge \log s(f_*),
\end{equation}
where $h(f)$ is the topological entropy of the mapping $f$ and $s(f_*)$ the spectral radius for the induced homomorphism $f_* \colon H_*(M;\R)\to H_*(M;\R)$ in homology. Recall that the \emph{topological entropy of a continuous mapping $f\colon M\to M$} is 
\[
	h(f) = \lim_{\varepsilon\to 0} \limsup_{k \to \infty} \frac{\log S_\varepsilon(k;f)}{k},
\]
where $S_\varepsilon(k;f)$ is the cardinality of a maximal $(k,\varepsilon;f)$-separated set in $M$; a set $E \subset M $ is \emph{$(k,\varepsilon;f)$-separated} if for $x,y\in E$, $x\ne y$, there exists $j\in \{1,\ldots, k\}$ for which points $f^j(x)$ and $f^j(y)$ have distance at least $\varepsilon$. We refer to Shub \cite{Shub-BAMS-1974} and Yomdin \cite{Yomdin-1987} for more detailed discussions on the terminology and the conjecture.

To our knowledge inequality \eqref{eq:Shub} is not known for uniformly quasiregular mappings. By Theorem \ref{alleigenvaluessamediag}, we have that a uniformly quasiregular mapping $f\colon M \to M$ satisfies
\[
	 s(f^* ) = \deg f ,
\]
which then implies the same for $s(f_*)$ by Poincar\'e duality. Thus, for uniformly quasiregular mappings the inequality \eqref{eq:Shub} is equivalent to the inequality 
\[
	h(f) \geq \log(\deg f).
\]

\subsection*{An application: the degree spectrum}

As an application of Theorem \ref{alleigenvaluessamediag}, we consider the \emph{degree spectrum of uniformly quasiregular mappings on a closed manifold $M$}, that is, the set of all degrees $\deg f$ of uniformly quasiregular self-maps $f\colon M \to M$ of $M$.

It is a simple corollary of the uniformly quasiregular Sto\"ilow theorem of Martin and Peltonen that the $n$-sphere admits uniformly quasiregular mappings of all degrees. In contrast to the case of spheres, in presence of non-trivial cohomology of order $k \in \{1, \ldots, n-1\}$ not all positive integers appear as degrees of uniformly quasiregular mappings. We have the following corollary of Theorem \ref{alleigenvaluessamediag}.

\begin{cor}\label{cor:degree_spectrum}
	Let $f\colon M\to M$ be a uniformly quasiregular self-map of a closed, connected, oriented Riemannian $n$-manifold $M$ and suppose that, for some $1\le k < n$, $k \dim H^k(M;\R)$ is coprime to $n$. Then $(\deg f)^{1/n}$ is an integer.
\end{cor}
\begin{proof}[Sketch of proof]
	Let $d=\dim H^k(M;\R)$. First, we observe that, by Theorem \ref{alleigenvaluessamediag}, the determinant of $f^* \colon H^k(M;\R)\to H^k(M;\R)$ is $\pm\left( \deg f \right)^{kd/n}$. On the other hand, by embedding $H^k(M;\Z)/\tor(H^k(M;\Z))$ into $H^k(M;\R)$ as an $f^*$-invariant subgroup, we find a basis of $H^k(M;\R)$ for which the matrix of $f^*$ has integer coefficients. Thus $\left(\deg f\right)^{kd/n} = \left|\det f^*\right|$ is an integer. Since $kd$ and $n$ are coprime, we conclude that $(\deg f)^{1/n}$ is an integer. A more detailed proof is given in Section \ref{sect:degree_limitation}.
\end{proof}

In the special case of products of spheres $M=\bS^{k_1}\times \cdots \times \bS^{k_p}$, Corollary \ref{cor:degree_spectrum} yields a sufficient condition for a characterization of the degree spectrum when combined with an existence theorem of Astola, Kangaslampi, and Peltonen \cite{Astola-Kangaslampi-Peltonen}; see also Mayer \cite{Mayer1997paper} for the case of Latt\`{e}s maps on spheres. 

\begin{cor}
	Let $M=\bS^{k_1}\times \cdots \times \bS^{k_p}$ and $n=k_1+ \cdots +k_p$. Suppose there exists $0 < k < n$ for which $k \dim H^{k}(M)$ is coprime to the dimension $n$. Then there exists a uniformly quasiregular mapping $M\to M$ of degree $d$ if and only if $d^{1/n} \in \Z_+$. In addition, every admissible degree is realized by a Latt\`es map. 
\end{cor}

Here the necessity of the degree condition follows from Corollary \ref{cor:degree_spectrum} as discussed above. The sufficiency follows from the aforementioned result of Astola, Kangaslampi, and Peltonen, which in this particular case states that for each $\lambda\in \Z_+$ there exists a uniformly quasiregular Latt\`es map $f \colon M \to M$ of degree $\lambda^n$. We refer to \cite{Astola-Kangaslampi-Peltonen} for a detailed discussion.

\subsection*{Outline of the proof of Theorem \ref{alleigenvaluessamediag}}
\enlargethispage{\baselineskip}

As the first step, we show that a quasiregular self-map $f\colon M\to M$ of a closed, connected, oriented Riemannian manifold induces a homomorphism on the conformal Sobolev--de Rham cohomology $H^*_{\CE}(M)$ of $M$; here the abbreviation ``CE'' stands for \emph{conformal exponent}. We define $H^*_{\CE}(M)$ as the cohomology of the complex $\cesobtloc(\wedge^* M)$, where for $k=1,\ldots, n-2$, the space $\cesobtloc(\wedge^k M)$ is the local Sobolev space $W^{d,\frac{n}{k}, \frac{n}{k+1}}_\loc(\wedge^k M)$ of conformal exponents, and for $k=0, n-1,n$, we replace the space $W^{d,\frac{n}{k},\frac{n}{k+1}}_\loc(\wedge^k M)$ by either a smaller or larger space of differential forms in order to obtain suitable complex. The precise definition of the complex $\cesobtloc(\wedge^* M)$ is given in Section \ref{sect:confcohom}.

We show that the obtained cohomology $H^*_\CE(M)$ is a sheaf cohomology which agrees with the singular cohomology $H^*(M;\R)$ on oriented Riemannian manifolds. This reduces Theorem \ref{alleigenvaluessamediag} to a corresponding statement on the homomorphism $f^* \colon H^k_\CE(M) \to H^k_\CE(M)$. Similar Sobolev-de Rham complexes and cohomologies have been considered by Donaldson and Sullivan \cite{DonaldsonSullivan1989paper} and Gol'dshtein and Troyanov \cite{GoldsteinTroyanov2010paper}; see also e.g.\ Bonk--Heinonen \cite{Bonk-Heinonen_Acta} for an application of conformally invariant Sobolev spaces.

The reason we consider the cohomology $H^*_{\CE}(M)$ instead of the standard de Rham cohomology $H^*_{\mathrm{dR}}(M)$ is that that a pull-back of a $C^\infty$-smooth form under a quasiregular mapping need not be $C^\infty$-smooth. Thus $f$ does not yield a natural pull-back operator on de Rham cohomology. However, the pull-back induced by $f$ induces a linear self-map of the partial Sobolev space $W^{d,\frac{n}{k},\frac{n}{k+1}}_{\text{loc}}(\wedge^k M)$ for $0 < k < n$. In the case of the complex $\cesobtloc(\wedge^* M)$, the pull-back extends to the whole complex by the higher integrability of quasiregular mappings.

After these preliminaries, the proof of Theorem \ref{alleigenvaluessamediag} consists of the following three steps. First we show that a proper quasiregular mapping induces a push-forward $f_* \colon \widetilde W^d_{\CE}(\wedge^* M) \to \widetilde W^d_{\CE}(\wedge^* M)$ in the conformal Sobolev chain complex and, \emph{a fortiori}, a corresponding homomorphism in the cohomology $H^*_\CE$; cf.\ Edmonds \cite{Edmonds-MMJ}.

\begin{thm}\label{thm:push_forward_cohomology}
	Let $f\colon M\to N$ be a quasiregular mapping between closed, connected, oriented Riemannian manifolds $M$ and $N$. Then there exists a (natural) linear map $f_* \colon H^*_\CE(M)\to H^*_\CE(N)$ satisfying $f_* f^* = (\deg f) \id$.
\end{thm}

We expect that the existence of the aforementioned push-forward operator for the conformal Sobolev complex is known to the experts. However, we have not found it discussed in the literature; see Heinonen--Kilpel\"ainen--Martio \cite[pp. 263-268]{HeinonenKilpelainenMartio2006book}, or \cite{OkuyamaPankka2014aper}, for the push-forward of functions.

The push-forward operator yields the following estimate of a cohomology class $c \in H^*_\CE(M)$; here $\norm{\cdot}_{n/k}$ is the conformally invariant norm
\[
	\norm{c}_\frac{n}{k} 
	= \inf_{\omega \in c} \left( \int_M |\omega|^\frac{n}{k} \vol_M \right)^\frac{k}{n}
\]
of cohomology classes in $H^k_\CE(M)$.

\begin{thm}\label{thm:estimate}
	Let $f\colon M \to M$ be a $K$-quasiregular self-map of a closed, connected, oriented Riemannian $n$-manifold $M$, and let $c\in H^k_\CE(M)$. Then there exists $C=C(n,K)$ for which 
	\[
		\frac{(\deg f)^\frac{k}{n}}{C} \norm{c}_\frac{n}{k} \le \norm{f^* c}_\frac{n}{k} \le C (\deg f)^\frac{k}{n} \norm{c}_\frac{n}{k}.
	\]
\end{thm}
The subtlety here is that, on the level of the chain complex, $f^*$ need not be surjective from the cohomology class $c$ to the cohomology class $f^*c$; this issue is addressed with the push-forward operator $f_*$ on the level of forms.

Having Theorem \ref{thm:estimate} at our disposal, we obtain Theorem \ref{alleigenvaluessamediag} by applying a complexification of Theorem \ref{thm:estimate} to the iterates of the uniformly quasiregular map $f$ and to an eigenvector class $c$ with eigenvalue $\lambda$. Since the iterates $f^m$ are $K$-quasiregular for $m \in \Z_+$, the constant $C$ of Theorem \ref{thm:estimate} is independent of $m$, which gives $\abs{\lambda} = (\deg f)^{k/n}$ in the limit. Diagonalizability is obtained by considering the Jordan decomposition of the matrix of the linear mapping $f^* \colon H^k_\CE(M)\to H^k_\CE(M)$.

\bigskip
\noindent
{\bf Acknowledgments.} We thank Marc Troyanov for discussions on conformal cohomology. The second author thanks Tuomas Sahlsten for discussions on the dynamics of uniformly quasiregular mappings. We would also like to thank the referee for kind remarks and suggestions improving the manuscript.


\section{Preliminaries}

Let $(M,\langle \cdot, \cdot \rangle)$ be a connected and oriented Riemannian $n$-manifold for $n\ge 2$, where $\langle \cdot, \cdot \rangle$ denotes the Riemannian metric of $M$. Let $\sigma \colon TM \to T^*M$ be the bundle isomorphism associated to the Riemannian metric, that is, $\sigma(v)(w) = \langle w,v \rangle$ for $x\in M$ and $w,v\in T_xM$. We denote the induced Riemannian metric on exterior powers $\wedge^k T^*M$ of the cotangent bundle $T^*M$ also by $\langle \cdot, \cdot \rangle$, that is, for each $x\in M$, $\langle \cdot, \cdot \rangle$ is the Grassmann inner product on $\wedge^k T_x^* M$ satisfying 
\[
	\langle \sigma(v_1)\wedge \cdots \wedge \sigma(v_k), \sigma(w_1)\wedge \cdots \wedge \sigma(w_k)\rangle 
	= \det \left(\langle v_i,w_j \rangle\right)_{ij}
\]
for all $v_1,\ldots, v_k,w_1,\ldots, w_k\in T_x M$. We denote by $|\cdot|\colon \wedge^k T^* M \to [0,\infty)$ the associated norm, $\omega \mapsto \langle \omega, \omega \rangle^{1/2}$, induced by the inner product $\langle \cdot, \cdot \rangle$. Let also $\vol_M$ be the \emph{volume form on $M$} determined by the Riemannian metric and compatible with the chosen orientation on $M$.

For each $k\in \N$, we denote by $\hodge \colon \wedge^k T^*M \to \wedge^{n-k} T^*M$ the (point-wise) \emph{Hodge star-operator on $M$} determined by 
\[
	\alpha \wedge \hodge \beta = \langle \alpha, \beta \rangle \vol_M(x)
\]
for each $\alpha, \beta \in \wedge^k T^*_xM$ and $x\in M$.

\subsection{Sobolev spaces of differential forms}

We now briefly recall the partial Sobolev spaces $W^{d,p,q}(\wedge^k M)$ on $M$; see Gol'dstein--Troyanov \cite{GoldsteinTroyanov2006paper}, Iwaniec--Lutoborski \cite{IwaniecLutoborski1993paper}, and Iwaniec--Scott--Stroffolini \cite{IwaniecScottStroffolini1999paper} for more details.

Given $k\in \N$, we call a measurable section $M \to \wedge^k T^*M$ a \emph{measurable $k$-form on $M$}; note that $\wedge^k T^*M$ is trivial for $k>n$. Further, given $p\in [1,\infty)$, we denote by $L^p(\wedge^k M)$ the space of all $p$-integrable $k$-forms, that is, the space of measurable $k$-forms $\omega \colon M \to \wedge^k T^*M$ for which the $L^p$-norm
\[
	\norm{\omega}_p = \left(\int_M  |\omega|^p \vol_M \right)^{1/p}
\]
is finite. The space $L^\infty(\wedge^k M)$ of essentially bounded $k$-forms is defined, as usual, to consist of those $k$-forms $\omega$ for which 
\[
	\norm{\omega}_\infty = \mathrm{esssup}_{x\in M} |\omega_x| <\infty.
\]

A $k$-form $\omega$ belongs to the local space $L^p_\loc(\wedge^k M)$ if the point-wise norm function $x\mapsto |\omega_x|$ is in the local space $L^p_\loc(M)$. We denote by $C^\infty(\wedge^k M)$ and $C^\infty_0(\wedge^k M)$ the spaces of smooth $k$-forms and compactly supported smooth $k$-forms on $M$, respectively.

Let $\omega \in L^p(\wedge^k M)$ for some $p \geq 1$. A measurable $(k+1)$-form $d\omega$ on $M$ is a \emph{weak exterior derivative of $\omega$} if, for all $\eta \in C_0^\infty(\wedge^{k+1}M)$, 
\[
	\int_M \left<d\omega, \eta\right> \vol_M = \int_M \left<\omega, d^*\eta\right> \vol_M,
\]
where $d^*$ is the \emph{coexterior derivative} $d^* = (-1)^{nk+1}\hodge d \hodge$. The weak exterior derivative $d\omega$ of $\omega$ is unique up to a set of measure zero. 

The \emph{partial Sobolev $(p,q)$-space $W^{d,p,q}(\wedge^k M)$ of $k$-forms on $M$} is the space of all $k$-forms $\omega \in L^p(\wedge^k M)$ having a weak exterior derivative $d\omega$ in $L^q(\wedge^{k+1} M)$. We call the space $W^{d,p,p}(\wedge^k M)$ the \emph{partial Sobolev space} and denote 
\[
	W^{d,p}(\wedge^k M) = W^{d,p,p}(\wedge^k M).
\]
We denote by $\norm{\cdot}_{d,p,q}$ the norm 
\[
	\norm{\omega}_{d,p,q} = \norm{\omega}_p + \norm{d\omega}_q
\]
on $W^{d,p,q}(\wedge^k M)$. As usual, the local space $W^{d, p, q}_\loc(\wedge^k M)$ is the space of $k$-forms in $L^p_\loc(\wedge^k M)$ having a weak exterior derivative in $L^q_\loc(\wedge^{k+1}M)$.

Similarly as the spaces $W^{d,p}(\wedge^k M)$, the spaces $W^{d,p,q}(\wedge^k M)$ are also complete for all $p,q\in [1,\infty]$. We start the proof with an auxiliary lemma.

\begin{lemma}\label{weakdbysmoothlimit}
	Let $p,q\in [1,\infty]$, $\omega \in L^p(\wedge^k M)$, and $\zeta \in L^q(\wedge^{k+1} M)$. Suppose that there exists a sequence $(\omega_i)$ in the space $W^{d, p, q}(\wedge^k M)$ for which $\omega_i \to \omega$ in $L^p(\wedge^k M)$, and $d\omega_i \to \zeta$ in $L^q(\wedge^{k+1}M)$. Then $\zeta$ is a weak exterior derivative of $\omega$.
\end{lemma}
\begin{proof}
	Let $\eta\in C^\infty_0(\wedge^{k+1} M)$ be a smooth test function. Then
	\begin{align*}
		&\abs{\int_M \langle \omega, d^*\eta\rangle\vol_M - \int_M \left<\zeta, \eta\right>\vol_M}\\
		&\qquad\leq \abs{\int_M \langle \omega-\omega_i, d^*\eta\rangle \vol_M}
			+ \abs{\int_M \langle \zeta - d\omega_i, \eta \rangle \vol_M}\\
		&\qquad\leq \norm{\omega-\omega_i}_{p}\norm{d^*\eta}_{p^*}
			+ \norm{\zeta - d\omega_i}_{q}\norm{\eta}_{q^*} \to 0
	\end{align*}
	as $i \to \infty$, where $p^*$ and $q^*$ are dual exponents of $p$ and $q$, respectively. Thus $\zeta = d\omega$, which concludes the proof. 
\end{proof}
\begin{lemma}\label{partialsobolevbanach}
	The partial Sobolev spaces $W^{d,p,q}(\wedge^k M)$ of $k$-forms on $M$ are Banach spaces for $p,q\in [1,\infty]$. 
\end{lemma}
\begin{proof}
	Let $p,q\in [1,\infty]$ and let $(\omega_i)$ be a Cauchy-sequence in the space $W^{d,p,q}(\wedge^k M)$. Then, by completeness of $L^p$-spaces, the sequences $(\omega_i)$ and $(d\omega_i)$ converge to forms $\omega$ and $\zeta$ in $L^p(\wedge^k M)$ and $L^q(\wedge^{k+1}M)$, respectively. Then by Lemma \ref{weakdbysmoothlimit}, $\zeta = d\omega$, and consequently $\omega \in W^{d,p,q}(\wedge^k M)$ and $\omega_i \to \omega$ in $W^{d,p,q}(\wedge^k M)$. Hence, $W^{d,p,q}(\wedge^k M)$ is complete.
\end{proof}

If $M$ is closed, by Iwaniec--Scott--Stroffolini \cite[Corollary 3.6]{IwaniecScottStroffolini1999paper}, the smooth forms $C^\infty(\wedge^k M)$ are dense in $W^{d,p}(\wedge^k M)$ for $p\in [1,\infty)$. The argument of Iwaniec--Scott--Stroffolini also yields that $C^\infty(\wedge^k M)$ is dense in $L^p(\wedge^k M)$ and $W^{d,p,q}(\wedge^k M)$ for all $p,q\in [1,\infty)$. We refer to \cite{IwaniecScottStroffolini1999paper} for details.

The Sobolev--Poincar\'e inequality of Gol'dshtein and Troyanov (see \cite[Theorem 1.1 and Appendix A]{GoldsteinTroyanov2006paper}) for $k$-forms on $M$ states that for  
\begin{equation}\label{eq:pqn}
	\frac{1}{q} - \frac{1}{p} \le \frac{1}{n} 
\end{equation}
and a $k$-form $\omega \in W^{d,p,q}(\wedge^k M)$ on a closed manifold $M$, there exists a constant $C = C(M,p,q)\ge 1$ and a closed $k$-form $\zeta$ for which 
\begin{equation}\label{eq:SP}
	\norm{\omega -\zeta}_p \le C \norm{d\omega}_q;
\end{equation}
see Iwaniec--Lutoborski \cite[Corollary 4.1]{IwaniecLutoborski1993paper} for the corresponding result in the closed Euclidean ball. We also use the fact that the dependence of $\zeta$ on $\omega$ is linear. It is crucial to note that the inequality is only given for $p, q \in (1, \infty)$.

An immediate corollary of this Sobolev--Poincar\'e inequality is that the image of the exterior derivative $d\colon W^{d,p,q}(\wedge^k M) \to L^q(\wedge^{k+1}M)$ is a closed subspace; the same holds in the case of a closed Euclidean ball.

\begin{cor}\label{lpqcompleteness}
	Let $M$ be a closed $n$-manifold and suppose that exponents $p, q \in (1, \infty)$ satisfy \eqref{eq:pqn}. Then the space $dW^{d, p, q}(\wedge^k M)$ is a closed subspace of $L^q(\wedge^{k+1}M)$.
\end{cor}
\begin{proof}
	Let $(\tau_i) = (d\omega_i)$ be a Cauchy-sequence, with respect to the $L^q$-norm, in $dW^{d, p, q}(\wedge^k M)$. Then $(\tau_i)$ converges to some $\tau\in L^q(\wedge^{k+1}M)$ in the $L^q$-norm. By the Sobolev--Poincar\'e inequality \eqref{eq:SP}, there exist closed $k$-forms $\zeta_i$ in $L^p(\wedge^k M)$ for which
	\begin{align*}
		\norm{(\omega_i -\zeta_i) - (\omega_j -\zeta_j)}_p
		&= \norm{(\omega_i -\omega_j) - (\zeta_i -\zeta_j)}_p\\
		&\leq C\norm{d(\omega_i - \omega_j)}_q
		= C\norm{\tau_i - \tau_j}_q.
	\end{align*}
	As such, the sequence $(\omega_i - \zeta_i)$ converges to some $\omega \in W^{d, p, q}(\wedge^k M)$. Since $d(\omega_i - \zeta_i) = \tau_i$ and $\omega_i - \zeta_i \to \omega$ in the $W^{d,p,q}$-norm as $i\to \infty$, we conclude that $\tau = d\omega$.
\end{proof}
	
\subsection{Quasiregular mappings and Sobolev forms}\label{subsect:qr}

We recall that a continuous mapping $f\colon M \to N$ between oriented Riemannian $n$-manifolds is \emph{$K$-quasiregular} for $K\ge 1$ if $f$ belongs to the local Sobolev space $W^{1,n}_\loc(M;N)$ of mappings $M\to N$ and satisfies the distortion estimate
\begin{equation}\label{eq:qrdef}
	\norm{Df}^n \le K J_f \quad \text{a.e.\ on}\ M,
\end{equation}
where $\norm{Df}$ is the operator norm of the weak differential $Df$ of $f$, defined by 
\[
	\norm{Df(x)} = \max_{\substack{v\in T_x M \\ |v|=1}} |Df(x) v|
\]
for almost every $x\in M$, and $J_f$ is the Jacobian determinant $\det Df$. For the basic properties of quasiregular mappings, see for example Rickman \cite{Rickman1993book} and Iwaniec--Martin \cite{IwaniecMartin2001book}.

In the above definition of quasiregular mappings, the local Sobolev space $W^{1,n}_\loc(M;N)$ is defined using an isometric embedding of the manifold $N$ into a Euclidean space; see e.g. Haj\l asz--Iwaniec--Mal\'y--Onninen \cite{HajlaszIwaniecMalyOnninen2008paper}. In short, we fix a smooth Nash embedding $\iota \colon N \to \R^m$, and say that a mapping $f\colon M\to N$ is in the local Sobolev space $W^{1,n}_\loc(M;N)$ if the coordinate functions of the map $\iota \circ f \colon M \to \R^m$ are in the Sobolev space $W^{1,n}_\loc(M)$. For $f \in W^{1,n}_\loc(M;N)$, the weak derivative $Df \colon TM \to TN$ is the measurable bundle map satisfying $D(\iota \circ f) = D\iota \circ Df$; the map $Df$ is unique up to a set of measure zero. This definition is independent on the choice of the embedding $\iota$; see e.g. \cite[Section 2]{ConventvanSchaftingen2016paper}.

The pre- and post-composition of a quasiregular map with a bilipschitz map is quasiregular. In particular, if $f\colon M\to N$ is a quasiregular map between Riemannian manifolds, for each $x\in M$ and each $\varepsilon>0$ there exist $(1+\varepsilon)$-bilipschitz charts $\phi \colon U \to \R^n$ and $\psi \colon V \to \R^n$ on the manifolds $M$ and $N$, respectively, for which $x\in U$, $f(x)\in V$, and the composition $\psi \circ f \circ \phi^{-1}\colon \phi U \to \R^n$ is $(1+\varepsilon)^{4n} K$-quasiregular. See e.g.\ Kangaslampi \cite[Section 2.3]{Kangaslampi-thesis} for more discussion. Therefore, if a local property of quasiregular maps between Euclidean domains is preserved under composition by a bilipschitz map, this property also holds for quasiregular maps between Riemannian manifolds.

For the forthcoming discussion, we also record a standard point-wise estimate for the pull-back of $k$-forms. Let $f \colon M \to N$ be a $K$-quasiregular mapping between $n$-manifolds and, for $0<k\le n$, let $\omega$ be a measurable $k$-form on $N$. Then the pull-back $f^*\omega$ is a well-defined measurable form, since quasiregular maps satisfy Lusin's condition (N); see e.g.\ Rickman \cite[I.4.14]{Rickman1993book}. Moreover, there exists  $C=C(n,k,K)$ for which
\begin{equation}\label{eq:pointwise}
	\frac{1}{C} (|\omega|^{n/k} \circ f)J_f \le |f^*\omega|^{n/k} \le C (|\omega|^{n/k}\circ f)J_f
\end{equation}
holds almost everywhere on $M$.

Due to the importance of this estimate to our results, we sketch its proof for the reader's convenience. For this, it is useful to recall the pointwise \emph{comass} norm for differential forms, given by
\[
	\abs{\omega_x}_{\mass} = \sup \{ \omega_x(v) : v \in \wedge^k T_x M \text{ is simple}, \abs{v} \leq 1 \}.
\] 
For more details on the comass norm, see e.g. Federer \cite[Section 1.8]{Federer1969book}. Since quasiregular mappings are differentiable almost everywhere, the pull-back at a point $x$ is given for almost every $x$ by $(f^* \omega)_x = \omega_{f(x)} \circ \wedge^k Df(x)$. If $v \in \wedge^k T_x M$ is simple, then clearly $(\wedge^k Df(x))(v)$ is simple and $\abs{(\wedge^k Df(x))(v)} \leq \norm{Df(x)}^k \abs{v}$. Hence, we obtain for almost every $x \in M$ the estimate $\abs{(f^* \omega)_x}_{\mass} \leq \norm{Df(x)}^k \abs{\omega_{f(x)}}_{\mass}$. Since $\abs{\cdot}$ and $\abs{\cdot}_{\mass}$ are comparable by a dimensional constant, see again \cite[Section 1.8]{Federer1969book}, the upper half of estimate \eqref{eq:pointwise} now follows from inequality \eqref{eq:qrdef}. 

For the lower bound, let $l(T)$ be the minimal dilatation of the linear operator $T$ between normed spaces, that is, $l(T) = \inf\{\norm{T(v)} \colon \norm{v} = 1 \}$. By quasiregularity of $f$, the differential $Df$ of $f$ is invertible almost everywhere in $M$, and hence $\omega_{f(x)} = (f^* \omega)_x \circ (\wedge^k Df(x))^{-1} = (f^* \omega)_x \circ \wedge^k (Df(x))^{-1}$ for almost every $x\in M$. The previous estimates yield  
\[
	\abs{\omega_{f(x)}} \leq L \abs{(f^* \omega)_x} \cdot \norm{(Df(x))^{-1}}^k 
	= L \abs{(f^* \omega)_x} \cdot l(Df(x))^{-k} 
\]
for some dimensional constant $L$, and the lower bound in \eqref{eq:pointwise} is now due to the estimate $J_f \leq K^{n-1} l(Df)^n$.

Recall that a mapping $f\colon X\to Y$ between topological spaces $X$ and $Y$ is \emph{proper} if the pre-image $f^{-1}E$ of every compact set $E\subset Y$ is compact. The (global) degree $\deg f$ of a proper mapping $f\colon M\to N$ between connected and oriented $n$-manifolds $M$ and $N$ is the unique integer satisfying $f^*c_N = (\deg f)c_M$, where $c_M\in H^n_c(M;\Z)$ and $c_N\in H^n_c(N;\Z)$ are the positive generators of the compactly supported $n$-th Alexander--Spanier cohomology of $M$ and $N$, respectively. Note that $\deg f = \sum_{x \in f^{-1}\{y\}} i(x, f)$ for all $y \in N$, where $i(x, f)$ is the local index of $f$ at $x$. We recall also that, for proper non-constant quasiregular mappings, the degree is always positive. For more information, we refer to e.g.\ \cite{Massey1978book} for the Alexander--Spanier cohomology, and to \cite{Rickman1973paper}, \cite[Chapter I.4]{Rickman1993book} for the local index theory of quasiregular maps. 

By a change of variables, \eqref{eq:pointwise} immediately yields the following integral estimate in the case of proper non-constant quasiregular mappings.

\begin{lemma}\label{qrnormestimate}
	Let $0<k\leq n$ and let $f \colon M \rightarrow N$ be a proper non-constant $K$-quasiregular mapping between oriented $n$-manifolds $M$ and $N$. Then there is a constant $C = C(n, k, K)$ for which
	\[
		\frac{1}{C} (\deg f) \int_N \abs{\omega}^\frac{n}{k} \vol_N
		\leq \int_M \abs{f^*\omega}^\frac{n}{k} \vol_M
		\leq C (\deg f) \int_N \abs{\omega}^\frac{n}{k} \vol_M
	\]
	for all $k$-forms $\omega \in L^{n/k}(\wedge^k N)$.
\end{lemma}


\section{Conformal cohomology}\label{sect:confcohom}
	
In this section, we discuss the conformal Sobolev cohomology theory we use in the proof of Theorem \ref{alleigenvaluessamediag}.

\subsection{Flat and sharp $L^p$}
	
Let $M$ be a Riemannian $n$-manifold, $n\in \Z_+$, and $0 \le k \le n$. The \emph{flat and sharp $L^p$-spaces} $L^{p, \flat}(\wedge^k M)$ and $L^{p, \sharp}(\wedge^k M)$, for $p\in [1,\infty]$, are
\[
	L^{p, \flat}(\wedge^k M) = \bigcap_{s \in [1, p)} L^{s}(\wedge^k M)
\]
and
\[
	L^{p, \sharp}(\wedge^k M) = \bigcup_{s \in (p, \infty]} L^{s}(\wedge^k M).
\]
We also define local variants $L^{p, \flat}_{\text{loc}}(\wedge^k M)$ and $L^{p, \sharp}_{\text{loc}}(\wedge^k M)$ as usual: A $k$-form $\omega$ belongs to $L^{p, \flat}_{\text{loc}}(\wedge^k M)$ if for every $x \in M$ there exists a neighborhood $V$ of $x$ for which $\omega|_V \in L^{p, \flat}(\wedge^k V)$. The space $L^{p, \sharp}_{\text{loc}}(\wedge^k M)$ is defined similarly.
	
If $M$ has finite measure, we have the inclusions $L^{p, \sharp}(\wedge^k M) \subset L^p(\wedge^k M) \subset L^{p, \flat}(\wedge^k M)$. While the flat and sharp $L^p$-spaces are vector spaces, they have no obvious norms. See, however, e.g.\ the grand $L^p$-spaces of Iwaniec--Sbordone \cite{IwaniecSbordone1992paper} which are normed subspaces of $L^{p, \flat}$-spaces.  
	
Our interest is primarily in the spaces $L^{1, \sharp}_{\text{loc}}(\wedge^n M)$ and $L^{\infty, \flat}_{\text{loc}}(\wedge^0 M)$, since they are preserved by quasiregular mappings. We recall the following consequence of the higher integrability of quasiregular  mappings.

\begin{lemma}\label{qrpull-backhigherint}
	Let $M$ and $N$ be oriented Riemannian $n$-manifolds and let $f \colon M \to N$ be a proper non-constant quasiregular map. Then, for every $x \in M$, there exist neighborhoods $U \subset M$ of $x$ and $V = fU \subset N$ of $f(x)$ for which the following condition holds: for every $p \in (1, \infty)$, there exist $s_0 \in (1, \infty)$ and $s_n\in (1,\infty)$ for which the maps $f^* \colon L^{s_0}(\wedge^0 V) \to L^p(\wedge^0 U)$ and $f^* \colon L^{p}(\wedge^n V) \to L^{s_n}(\wedge^n U)$ are continuous. 
\end{lemma}

\begin{proof}
	By Martio \cite{Martio1975paper}, there exists $r > 1$ for which $J_f \in L^r_\loc(M)$; see also Meyers--Elcrat \cite{MeyersElcrat1975paper}. In addition to this, there exists $\eps > 0$ for which $J_f^{-\eps} \in L^1_\loc(M)$; see e.g.\ Hencl-Koskela-Zhong \cite{HenclKoskelaZhong2007paper} for a far-reaching discussion.
	
	Let $U$ be a normal neighborhood of $f$ at $x$, that is, a pre-compact neighborhood $U$ of $x$ for which $\partial fU = f(\partial U)$. Then $f|_U \colon U \to fU$ is proper; see e.g.\ V\"ais\"al\"a \cite[Lemma 5.1]{Vaisala1966paper}. We set $V=fU$. Since $U \subset M$ is precompact, we have $J_f \in L^r(U)$ and $J_f^{-\eps} \in L^1(U)$.
	
	Let $p \in (1, \infty)$, $\omega \in L^{p}(\wedge^n V)$ and $q = \left(r + p - 1\right)/r > 1$. Then $p > q$, and by Hölder's inequality and the change of variables,
	\begin{align}\label{eq:higherpullbackn}\begin{split}
		\int_U \abs{f^*\omega}^\frac{p}{q} 
		&\leq C\int_U \left(\abs{\omega}^\frac{p}{q} \circ f\right) J_f^\frac{p}{q}
		= C\int_U \left(\abs{\omega}^\frac{p}{q} \circ f\right) J_f^\frac{1}{q} 
				\cdot J_f^{\frac{p-1}{q}} \\
		&\leq C\left(\int_U \left(\abs{\omega}^p \circ f\right) J_f \right)^\frac{1}{q} 
			\left(\int_U J_f^{\frac{p-1}{q-1}} \right)^\frac{q-1}{q}\\
		&= C (\deg f)^\frac{1}{q} \norm{\omega}_p^\frac{p}{q} 
			\left(\int_U J_f^{r}\right)^\frac{q-1}{q}.
	\end{split}\end{align}
	Thus $f^* \colon L^{p}(\wedge^n V) \to L^{p/q}(\wedge^n U)$ is well-defined and continuous.
	
	Next, let $\omega$ be a 0-form of $V$, and let $q = 1 + 1/\eps$. Since $f^*\omega = \omega \circ f$, we have
	\begin{align}\label{eq:higherpullbackzero}\begin{split}
		\int_U \abs{f^*\omega}^p 
		&= \int_U \abs{\omega}^p \circ f 
		= \int_U \left(\abs{\omega}^p \circ f\right) 
			J_f^\frac{1}{q} \cdot J_f^{-\frac{1}{q}}\\
		&\leq \left( \int_U \left(\abs{\omega}^{pq} \circ f\right) J_f \right)^\frac{1}{q}
			\left( \int_U J_f^{-\frac{1}{q-1}} \right)^\frac{q-1}{q}\\
		&= (\deg f)^\frac{1}{q} \norm{\omega}_{pq}^p 
			\left( \int_U J_f^{-\eps} \right)^\frac{q-1}{q}.
	\end{split}\end{align}
	The continuity of $f^* \colon L^{pq}(\wedge^0 V) \to L^p(\wedge^0 U)$ for every $p \in (1, \infty)$ now follows.
\end{proof}

As an immediate consequence of Lemma \ref{qrpull-backhigherint}, we obtain that the pull-back preserves the local sharp and flat spaces $L^{1,\#}_\loc(\wedge^n M)$ and $L^{\infty,\flat}_\loc(\wedge^0 M)$. More precisely, we have the following corollary.
\begin{cor}\label{qrflatandsharp}
	Let $M$ and $N$ be oriented Riemannian $n$-manifolds, and let $f \colon M \to N$ be a proper non-constant quasiregular map. The pull-back $f^*$ maps $L^{1, \sharp}_{\text{loc}}(\wedge^n N)$ into $L^{1, \sharp}_{\text{loc}}(\wedge^n M)$ and $L^{\infty, \flat}_{\text{loc}}(\wedge^0 N)$ into $L^{\infty, \flat}_{\text{loc}}(\wedge^0 M)$.
\end{cor}
\begin{proof}
	Let $\omega \in L^{1, \sharp}_{\text{loc}}(\wedge^n N)$, let $x \in M$, and let $U$, $V$ be as in Lemma \ref{qrpull-backhigherint}. Then there exists a precompact neighborhood $W \subset V$ of $f(x)$ for which $\omega \vert_W \in L^{1, \sharp}(\wedge^n W)$, and hence there exists $p > 1$ for which $\omega \vert_W \in L^{p}(\wedge^n W)$. Let $U' = f^{-1}W \cap U$. Since $W$ is precompact, we may assume $p < \infty$. Then, by Lemma \ref{qrpull-backhigherint}, there exists $s_n > 1$ for which $f^*\omega \vert_{U'} \in L^{s_n}(\wedge^n U')$. Hence $f^*\omega \vert_{U'} \in L^{1, \sharp}(\wedge^n U')$. This completes the proof that $f^*\omega \in L^{1, \sharp}_{\text{loc}}(\wedge^n M)$.
	
	Similarly, let $\omega \in L^{\infty, \flat}_{\text{loc}}(\wedge^0 N)$, let $x \in M$, and let $U$, $V$ be as in Lemma \ref{qrpull-backhigherint}. Then there exists a precompact neighborhood $W \subset V$ of $f(x)$ for which $\omega \vert_W \in L^{\infty, \flat}(\wedge^0 W)$. Let $p \in (1, \infty)$, let $s_0$ be as in Lemma \ref{qrpull-backhigherint}, and again let $U' = f^{-1}W \cap U$. Since $\omega\vert_W \in L^{s_0}(\wedge^0 W)$, we have $f^*\omega\vert_{U'} \in L^p(\wedge^0 U')$. Note that this holds also for $p = 1$, since $U'$ is precompact due to properness of $f$. We obtain that $f^*\omega\vert_{U'} \in L^{\infty, \flat}(\wedge^0 U')$, which concludes the proof.
\end{proof}
	
\subsection{Conformal Sobolev spaces}\label{sect:confsob}
	
The \emph{weak Sobolev space of conformal exponents} $\cesob(\wedge^k M)$ is defined as 
\[
	\cesob(\wedge^k M) = \left\{\begin{array}{ll}
		W^{d, \infty, n}(\wedge^0 M), & \text{for}\ k=0, \\
		W^{d, \frac{n}{k}, \frac{n}{k+1}}(\wedge^k M), & \text{for}\ k\in \{1,\ldots, n-1\}, \\
		L^1(\wedge^n M), & \text{for}\ k = n, \\
		0, & \text{for}\ k>n.
	\end{array}\right.
\]
The local space $\cesobloc(\wedge^k M)$ is defined analogously. Since each form in $d\cesob(\wedge^k M)$ has a vanishing weak exterior derivative, we have \linebreak $d\cesob(\wedge^k M) \subset \cesob(\wedge^{k+1} M)$. Thus the sequence
\[
	\xymatrix{
		\cdots \ar[r] 
		& \cesob(\wedge^{k-1} M) \ar[r]^{d} 
		& \cesob(\wedge^{k} M) \ar[r]^{d} 
		& \cesob(\wedge^{k+1} M) \ar[r] 
		& \cdots
	}
\]
is a chain complex. We have a similar complex for $\cesobloc$.
	
The abbreviation ``OCE'' stands for \emph{original conformal exponent}. The $\cesobloc$-complex is discussed by Donaldson and Sullivan in \cite{DonaldsonSullivan1989paper}, and the $\cesob$-complex alongside its cohomology spaces by Gol'dshtein and Troyanov in \cite{GoldsteinTroyanov2010paper}. Gol'dshtein and Troyanov show that on closed manifolds, the $k$-th cohomology of the $\cesob$-complex agrees with $k$-th real singular cohomology for $k \in \{2, \ldots, n-1\}$. They also provide a counterexample that, for $k=1$, the cohomologies are not necessarily isomorphic.

We consider a modification $\cesobt(\wedge^* M)$ of the complex $\cesob(\wedge^* M)$ given by
\begin{align*}
	\cesobt(\wedge^0 M) 
		&= \left\{ \omega \in L^{\infty, \flat}(\wedge^0 M) \;\big\vert\; d\omega \in L^n(\wedge^1 M) \right\},\\
	\cesobt(\wedge^k M) 
		&= \cesob(\wedge^k M) \quad \text{for}\ k=1,\ldots, n-2, \\
	\cesobt(\wedge^{n-1} M) 
		&= \big\{ \omega \in L^{\frac{n}{n-1}}(\wedge^{n-1} M) \;\big\vert\;d\omega \in L^{1, \sharp}(\wedge^n M) \big\},\; \text{and}\\
	\cesobt(\wedge^n M) 
		&= L^{1, \sharp}(\wedge^n M).
\end{align*}
Heuristically, we flatten the $L^\infty$-space of $0$-forms and sharpen the $L^1$-space of $n$-forms. The local spaces $\cesobtloc(\wedge^{k} M)$ are defined analogously. 

We note here that in fact $\cesobtloc(\wedge^0 M) = W^{d, n}_\loc(\wedge^0 M)$. Indeed, clearly $\cesobtloc(\wedge^0 M) \subset W^{d, n}_\loc(\wedge^0 M) = W^{1, n}_\loc(M; \R)$, where $W^{1, n}_\loc(M; \R)$ is the classical first order local $n$-Sobolev space of measurable real functions. For the converse direction, let $q \in (n, \infty)$, and let $p^{-1} = q^{-1} + n^{-1}$. Since $1 < p < n$, we have $W^{1,n}_\loc(M; \R) \subset W^{1,p}_\loc(M; \R)$. Furthermore, by the classical Sobolev embedding theorem, $W^{1,p}_\loc(M; \R) \subset L^q_\loc(M)$; see e.g.\ \cite[Theorem 5.4]{Adams1975book}, which we apply using the fact that $u \colon M \to \R$ is locally $s$-Sobolev for $1 < s < \infty$ if and only if $u \circ \psi^{-1}$ is locally $s$-Sobolev for every smooth bilipschitz chart $\psi$ on $M$. We conclude that $W^{1,n}_\loc(M; \R) \subset \cesobtloc(\wedge^0 M)$. Note also that $\cesobt(\wedge^{0} M) = W^{d,n}(\wedge^0 M)$ immediately follows whenever $M$ is closed.
	
As previously, we obtain chain complexes $\cesobt(\wedge^* M)$, and $\cesobtloc(\wedge^* M)$. We denote by $\cehom{*}$ the cohomology of the $\cesobtloc$-complex, that is,
\[
	\cehom{k}(M) = \dfrac{
		\ker\left(d\colon \cesobtloc(\wedge^{k} M) \to \cesobtloc(\wedge^{k+1} M)\right)
	}{
		\im\left(d\colon \cesobtloc(\wedge^{k-1} M) \to \cesobtloc(\wedge^{k} M)\right)
	}.
\] 
On a closed manifold $M$, the $\cesobt$-complex yields the same cohomology $\cehom{*}(M)$. Using the local complex has, however, the advantage that it simplifies the sheaf-theoretic proof that $\cehom{*}$ coincides with real singular cohomology.

For $0 < k < n$, Corollary \ref{lpqcompleteness} shows that the spaces $d\cesobt(\wedge^{k-1} M)$ are complete for closed $M$, where the case $k=1$ is due to the aforementioned fact that $\cesobt(\wedge^{0} M) = W^{d,n}(\wedge^0 M)$. We record this observation as a lemma.
	
\begin{lemma}\label{confexpcompleteness}
	Let $M$ be a closed Riemannian $n$-manifold and let $0 < k < n$. Then $d\cesobt(\wedge^{k-1} M)$ is complete under the $L^{n/k}$-norm.
\end{lemma}
	
A proper non-constant quasiregular map $f \colon M \to N$ between oriented $n$-manifolds induces a chain map $f^* \colon \cesobloc(\wedge^* N) \to \cesobloc(\wedge^* M)$ satisfying $f^* \circ d = d \circ f^*$; see \cite[Lemma 2.22]{DonaldsonSullivan1989paper} or \cite[Theorem 6.6]{GoldsteinTroyanov2010paper}, where the proofs are given for quasiconformal maps but the argument extends to the quasiregular case, or alternatively \cite[Lemma 3.6]{IwaniecMartin1993paper} or \cite[Section 15.3]{IwaniecMartin2001book}. We show that the same holds for the $\cesobtloc$-complex.
	
\begin{lemma}\label{qrchainmap}
Let $f \colon M \to N$ be a proper non-constant quasiregular mapping, where $M$ and $N$ are oriented Riemannian $n$-manifolds. Then $f$ induces a chain map $f^* \colon \cesobtloc(\wedge^* N) \to \cesobtloc(\wedge^* M)$ satisfying $f^* \circ d = d \circ f^*$, and consequently induces a linear map $f^* \colon \cehom{*}(N) \to \cehom{*}(M)$ on cohomology.
\end{lemma}
	
\begin{proof}
Let $\omega \in \cesobtloc(\wedge^k N)$. We consider first the case $k>0$. Since $\cesobtloc(\wedge^k N) \subset \cesobloc(\wedge^k N)$ and $f^*$ is a chain map between the $\cesobloc$-complexes, we have that $f^*\omega \in \cesobloc(\wedge^k M)$ and $df^*\omega = f^*d\omega$. Thus, by Corollary \ref{qrflatandsharp}, we have $f^*\omega \in \cesobtloc(\wedge^k M)$.
		
It remains to consider the case $k = 0$. In this case, $\omega$ is a Sobolev function and we may identify the weak exterior derivative with the weak gradient of the function. Hence, by for example \cite[Theorem 14.28.]{HeinonenKilpelainenMartio2006book}, $f^*\omega = \omega \circ f \in W^{1,n}_\text{loc}(M)$ and $df^*\omega = f^*d\omega$. Now, by Corollary \ref{qrflatandsharp}, we obtain that $f^*\omega \in \cesobtloc(\wedge^0 M)$. 
\end{proof}


\section{Equivalence of cohomologies}

In this section we show that the Sobolev--de~Rham cohomology $\cehom{*}(M)$ of an oriented Riemannian manifold $M$ is naturally isomorphic to the real singular cohomology $H^*(M;\R)$ of $M$.

\begin{thm}\label{derhamsobolevequivalence}
	For the category of oriented Riemannian manifolds and proper non-constant quasiregular mappings, there is a natural isomorphism from the singular cohomology to the Sobolev--de~Rham cohomology in the following sense: For each Riemannian manifold $M$ there exists an isomorphism $\nu_M \colon H^*(M;\R) \to \cehom{*}(M)$ having the property that, for a proper non-constant quasiregular mapping $f\colon M\to N$ between oriented Riemannian $n$-manifolds $M$ and $N$, the diagram 
	\begin{equation}\label{eq:same_cohom_maps}
		\xymatrix{
			H^*(N; \R) \ar[r]^{\nu_N^*} \ar[d]^{f^*} & \cehom{*}(N) \ar[d]^{f^*}\\
			H^*(M; \R) \ar[r]^{\nu_M^*} & \cehom{*}(M)
		}
	\end{equation}
	commutes.
\end{thm}
Having this isomorphism of cohomology theories at our disposal, we may identify the linear maps $f^* \colon H^*(N;\R) \to H^*(M;\R)$ and $f^* \colon \cehom{*}(N) \to \cehom{*}(M)$, and reduce the proof of Theorem \ref{alleigenvaluessamediag} to the corresponding question on the eigenvalues of $f^*\colon \cehom{*}(M) \to \cehom{*}(M)$.

The proof of Theorem \ref{derhamsobolevequivalence} is a variant of the sheaf theoretic proof of the de Rham theorem. The key ingredient for the proof is a Poincar\'e lemma for the conformal complex $\cesobtloc(\wedge^*M)$, which follows from the Sobolev-Poincar\'e inequality \eqref{eq:SP} for Euclidean balls; see also Iwaniec--Lutoborski \cite[Proposition 4.1]{IwaniecLutoborski1993paper}.

\begin{lemma}[Poincar\'e lemma for $W^{d, p, q}$]\label{sobolevpoincare} 
	Let $n\ge 2$ and let $p, q \in (1, \infty)$ be constants for which 
	\[
		\frac{1}{q} - \frac{1}{p} \leq \frac{1}{n}.
	\]
	Let $U$ be a domain in $\R^n$, $k \in \{1, \ldots, n\}$, and let $\omega \in L^q(\wedge^k U)$ be a weakly closed form. Then, for each $y \in U$, there exists a neighborhood $V \subset U$ of $y$ and a form $\tau \in W^{d, p, q}(\wedge^{k-1} V)$ for which $d\tau = \omega\vert_V$.
\end{lemma}

\begin{proof}
	Let $\omega \in L^q(\wedge^k U)$ be a $k$-form for which $d\omega = 0$. Let $B \subset \overline{B} \subset U$ be an open ball containing $y$, and let $(\eta_i)$ be a sequence of standard mollifiers. We may approximate $\omega\vert_B$ in $L^q(\wedge^k B)$ with smooth forms $\omega_i = \eta_i \ast \omega$. Now, $d\omega_i = \eta_i \ast d\omega = 0$ for each $i$.
	
	Since the forms $\omega_i$ are smooth and closed, the ordinary Poincar\'e lemma from de Rham theory yields forms $\tau_i \in C^\infty(\wedge^{k-1} B)$ satisfying $d\tau_i = \omega_i$. For each $i$, let $\tau'_i = \tau_i - \zeta_i$ where $\zeta_i$ is given by the Sobolev--Poincar\'e inequality \eqref{eq:SP} used on $\tau_i$, and note that $d\tau'_i = \omega_i$ for every $i$. But now, since $\zeta_i$ depend linearly on $\tau_i$, the sequence $(\tau'_i)$ is Cauchy in $W^{d, p, q}(\wedge^{k-1} B)$ and therefore has a limit $\tau' \in W^{d, p, q}(\wedge^{k-1} B)$ by Lemma \ref{partialsobolevbanach}. Since $d\tau' = \lim_{i \to \infty} d\tau_i' = \lim_{i \to \infty} \omega_i = \omega$, the proof is concluded.
\end{proof}

\begin{cor}[Poincar\'e lemma for $\cesobt$]\label{cesobolevpoincare}
	Let $U$ be a domain in $\R^n$, $n \geq 2$. Let $k \in \{1, \ldots, n\}$, and let $\omega \in \cesobt(\wedge^k U)$ be weakly closed. Then, for each $y \in U$, there exists a neighborhood $V \subset U$ of $y$ and a form $\tau \in \cesobt(\wedge^{k-1} V)$ for which $d\tau = \omega\vert_V$.
\end{cor}

\begin{proof}
	The claim follows immediately from Lemma \ref{sobolevpoincare}, where the case $k = 1$ is due to the fact that $\cesobt(\wedge^0 M) = W^{d, n}(\wedge^0 M)$.
\end{proof}

With Corollary \ref{cesobolevpoincare}, the proof of Theorem \ref{derhamsobolevequivalence} is for the most part straightforward for a reader familiar with sheaf theory. We nonetheless present a more detailed outline of the proof for the reader's convenience. Our presentation is based on Wells \cite[Chapter II]{Wells1980book} and Warner \cite[Chapter 5]{Warner1983book}. For the naturality of the induced homomorphism, our reference is Bredon \cite[Section II.8]{Bredon1997book}.

\subsection{Notation and terminology}

\subsubsection{The sheaves $\cesheaf{*}{\cdot}$}

We restrict our discussion to the particular case of sheaves of vector spaces, and refer to e.g.\ \cite{Bredon1997book} and \cite[Chapter II]{Wells1980book} for more general expositions on sheaf theory.

Let $M$ be an oriented Riemannian $n$-manifold. The presheaves $\cesheaf{*}{M}$ are defined by 
\[
	\cesheaf{*}{M} = \left\{ 
		\begin{array}{rcl} 
			U &\mapsto& \cesobtloc(\wedge^* U)\\ 
			i_{U, V} \colon U \hookrightarrow V 
				&\mapsto& i_{U, V}^* \colon \cesobtloc(\wedge^* V) \to \cesobtloc(\wedge^* U) 
		\end{array} \right\}, 
\]
where $i_{U, V}^*$ is the pullback map induced by the inclusion map $i_{U, V} \colon U \hookrightarrow V$. Recall that, more generally, a \emph{presheaf $\sheaf$ on $M$} is a contravariant functor from the category of open subsets of $M$ and inclusion maps to the category of vector spaces, that is, $\sheaf$ assigns to an open set $U \subset M$ a vector space $\sheaf(U)$, and to every inclusion $i_{U, V} \colon U \hookrightarrow V$ of open sets in $M$ a linear \emph{restriction homomorphism} $\sheaf(i_{U, V}) \colon \sheaf(V) \to \sheaf(U)$.

A presheaf $\sheaf$ on $M$ is called a \emph{sheaf} if, for every collection $\mathcal{U}$ of open subsets of $M$ with union $U_{\mathcal{U}} \subset M$, the following conditions are satisfied:
\begin{enumerate}
	\item\label{sheafcondition1} If for $v, w \in \sheaf(U_{\mathcal{U}} )$ the restrictions $\sheaf(i_{V, U_{\mathcal{U}} })v$ and $\sheaf(i_{V, U_{\mathcal{U}} })w$ agree for every $V \in \mathcal{U}$, then $v = w$.
	\item\label{sheafcondition2} If $v_{V} \in \sheaf(V )$ for every $V \in \mathcal{U}$ and the restrictions $\sheaf(i_{V \cap W, V})v_{V}$ and $\sheaf(i_{V \cap W, W})v_{W}$ agree whenever $V, W \in \mathcal{U}$ and $V \cap W \neq \emptyset$, then there exists $v \in \sheaf(U_{\mathcal{U}} )$ for which $\sheaf(i_{V, U_{\mathcal{U}}})v = v_{V}$ for every $V \in \mathcal{U}$.
\end{enumerate}
It is easily seen that the presheaves $\cesheaf{*}{M}$ are sheaves. Note here the crucial subtlety that we defined $\cesheaf{*}{M}$ using local spaces $\cesobtloc(\wedge^* U)$ instead of global ones $\cesobt(\wedge^* U)$: were $\cesheaf{*}{M}$ defined using $\cesobt(\wedge^* U)$ instead, the collated element $v$ of condition \eqref{sheafcondition2} would not necessarily satisfy the required global integrability for infinite collections $\mathcal{U}$.

\subsubsection{\'Etal\'e spaces and generated sheaves}

Let $\ceetale{*}{M}$ denote the associated \'etal\'e spaces of the sheaves $\cesheaf{*}{M}$. Recall that a presheaf $\sheaf$ over $M$ has an associated \emph{\'etal\'e space} $\Fetale$, which is a topological space together with a local homeomorphism $\pi \colon \Fetale \to M$, for which $\pi^{-1}\{x\}$ is a vector space for every $x \in M$  and the maps
\begin{align*}
	(f_1, f_2) &\mapsto f_1 - f_2, && (f_1, f_2) \in \{(f, g) \in \Fetale \times \Fetale \colon  \pi(f) = \pi(g)\},\\
	f_1 &\mapsto kf_1, && f_1 \in \Fetale, k \in \R. 
\end{align*}
are continuous.

The construction of the associated \'etal\'e spaces $\ceetale{*}{M}$ is by considering spaces of germs. Given an open set $U \subset M$, $\omega\in \cesobtloc(\wedge^* U)$, and $x\in U$, the \emph{germ $[\omega]_x$ of $\omega$ at $x$} is the equivalence class of forms $\omega' \in \cesobtloc(\wedge^* V)$, where $V$ is a neighborhood of $x$, for which $\omega|_W = \omega'|_W$ almost everywhere in a neighborhood $W \subset U \cap V$ of $x$. 

The \emph{stalk $(\cesheaf{*}{M})_x$ of $\cesheaf{*}{M}$ over $x\in M$} is the vector space of germs of $\cesheaf{*}{M}$ at $x$. The associated \'etal\'e space $\ceetale{*}{M}$ is then defined as a union of all stalks of $\cesheaf{*}{M}$, with topology generated by the sets $\{[\omega]_x \colon x\in U\}$ where $U\subset M$ is open and $\omega\in \cesobtloc(\wedge^* U)$. 

For each open set $U\subset M$, we also denote by $\Gamma(U, \ceetale{*}{M})$ the space of sections of the \'etal\'e space $\ceetale{*}{M}$ over $U$, that is, the vector space of continuous mappings $s \colon U \to \ceetale{*}{M}$ satisfying $\pi \circ s = \id_U$. The collection of vector spaces $\Gamma(U, \ceetale{*}{M})$ for each open set $U \subset M$, together with the natural restriction maps $\Gamma(V, \ceetale{*}{M}) \to \Gamma(U, \ceetale{*}{M})$, is a sheaf on $M$, called the \emph{sheaf of sections of the \'etal\'e space $\ceetale{*}{M}$} or alternatively the \emph{generated sheaf of the presheaf $\cesheaf{*}{M}$}. We denote by $\Gamma(\ceetale{*}{M})$ the sheaf of sections of $\ceetale{*}{M}$.

A \emph{presheaf homomorphism $\varphi\colon \sheaf \to \sheaf'$} between presheaves $\sheaf$ and $\sheaf'$ over the same space $M$ is a natural transformation from $\sheaf$ to $\sheaf'$, that is a collection 
\[
	\left\{\varphi_{U} \colon \sheaf(U) \to \sheaf'(U) \colon U \subset M \text{ is open}\right\}
\]
of linear maps, which commute with the restriction homomorphisms. Note that a presheaf homomorphism $\varphi \colon \sheaf \to \sheaf'$ induces for every $x \in M$ a linear map $\varphi_x \colon \sheaf_x \to \sheaf_x'$ between stalks over $x$. If every map $\varphi_U$ of a presheaf homomorphism $\varphi \colon \sheaf \to \sheaf'$ is bijective, then $\varphi$ is a \emph{presheaf isomorphism}. The terms \emph{sheaf homomorphism} and \emph{sheaf isomorphism} are also used when the domain and target presheaves are sheaves.

Since $\cesheaf{*}{M}$ is a sheaf, it is naturally isomorphic to its generated sheaf $\Gamma(\ceetale{*}{M})$; see \cite[Theorem II.2.2]{Wells1980book} or \cite[Proposition 5.8]{Warner1983book}. The explicit sheaf isomorphism is given by
\begin{align*}
	\varphi_{*} \colon \cesheaf{*}{M} \to \Gamma(\ceetale{*}{M}), \quad (\varphi_{*})_U(\omega) = (x\mapsto [\omega]_x),
\end{align*}
where $U \subset M$ is open and $\omega \in \cesobtloc(\wedge^* U)$.
\begin{rem}
	In some sources such as \cite{Warner1983book} and \cite{Bredon1997book}, the term \emph{complete presheaf} is used for sheaves instead, while the term \emph{sheaf} is used for the \'etal\'e spaces of presheaves. We follow here the terminology used in Wells \cite{Wells1980book}.
\end{rem}

\subsection{Fine resolution of the constant sheaf}

A sheaf $\sheaf$ over $M$ is \emph{fine} if every locally finite open cover $\{ U_i \colon i \in I\}$ of $M$ has a \emph{subordinate partition of unity}, that is, a collection 
\[
	\{ \lambda_i \colon \sheaf \to \sheaf : i \in I,\ \spt \lambda_i \subset U_i\}
\] 
of sheaf homomorphisms satisfying $\sum_{i\in I}\lambda_i = \id_{\sheaf}$. Note that the identity morphism $\id_{\sheaf}$ is defined by setting $(\id_{\sheaf})_U = \id_{\sheaf(U)}$ for every open $U$, and the support $\spt \lambda_i$ is the collection of points $x \in M$ that do not have a neighborhood $U$ where $(\lambda_i)_U$ is the zero map. 

Note that the infinite sum $\sum_{i\in I}\lambda_i$ is a sheaf homomorphism by local finiteness of the family of supports of $\lambda_i$. Indeed, for every $x \in M$, any sufficiently small neighborhood $U$ meets the support of only finitely many $\lambda_i$ due to local finiteness, and for all other $\lambda_i$ we have $(\lambda_i)_U = 0$. This yields a well-defined $(\sum_{i\in I}\lambda_i)_U$ for small enough open $U \subset M$, which can be extended to all open $U \subset M$ by taking unions of small $U$ and using conditions \eqref{sheafcondition1} and \eqref{sheafcondition2} in the definition of a sheaf.

We denote by $\Rsheaf_M$ the constant sheaf on $M$, which maps every open $U \subset M$ to the space of locally constant functions on $U$, and every inclusion $i_{U, V} \colon U \to V$ to the usual restriction map of functions. Note that the functions of $\Rsheaf_M(U)$ are constant on the components of $U$, and every stalk of $\Rsheaf_M$ is naturally isomorphic to $\R$. We obtain a natural inclusion sheaf homomorphism $i \colon \Rsheaf_M \to \cesheaf{0}{M}$. Furthermore, the weak exterior derivative $d$ induces a sheaf homomorphism $d \colon \cesheaf{k}{M} \to \cesheaf{k+1}{M}$ for every $k \in \N$. In what follows, we show that the sheaves $\cesheaf{*}{M}$ together with the sheaf homomorphisms $d$ form a \emph{fine resolution} of the constant sheaf $\Rsheaf_M$.

\begin{prop}\label{prop:finetorsionlessresolution}
	The sequence 
	\begin{equation*}\label{eq:sheafseq}
		\xymatrix{
			0 \ar[r]
			&\Rsheaf_M \ar[r]^-{i} 
			&\cesheaf{0}{M} \ar[r]^-{d}
			&\cesheaf{1}{M} \ar[r]^-{d}
			&\cdots
		}
	\end{equation*}
	of sheaves is a fine resolution of $\Rsheaf_M$, that is, the spaces $\cesheaf{*}{M}$ are fine and for every $x \in M$ the induced sequence
	\begin{equation}\label{eq:stalkseq}
		\xymatrix{
			0 \ar[r] 
			& \R \ar[r]^-{i_x} 
			&(\cesheaf{0}{M})_x \ar[r]^-{d_x} 
			& (\cesheaf{1}{M})_x \ar[r]^-{d_x} 
			&\cdots
		}
	\end{equation}
	on stalks is exact.
\end{prop}

\begin{proof}
	Let $k \in \N$, let $\mathcal{U} = \{U_i\}$ be a locally finite open cover of $M$, and let $\{\phi_i\}$ be a smooth partition of unity subordinate to $\mathcal{U}$. For each index $i$, let $\lambda_{i} \colon \cesheaf{k}{M} \to \cesheaf{k}{M}$ be the sheaf endomorphism given by $\omega \mapsto (\phi_i \vert_U) \omega$ for $\omega \in \cesobtloc(\wedge^k U)$. We obtain a partition of unity for $\cesheaf{k}{M}$ subordinate to $\mathcal{U}$. Hence the sheaf $\cesheaf{k}{M}$ is fine.
		
	It remains to verify the exactness of \eqref{eq:stalkseq} for every $x \in M$. Exactness at $\R$ follows since the maps $i_x$ are injective. Exactness at $(\cesheaf{k}{M})_x$ for $k > 0$ follows from the version of the Poincar\'e lemma in Corollary \ref{cesobolevpoincare}, which yields for every open nonempty $U \subset M$ that a locally closed $k$-form $\omega \in \cesobt(\wedge^k U)$ is locally exact. 

	For the remaining case $k=0$, let $x \in M$, $U$ a neighborhood of $x$, and let $u \in \cesobtloc(\wedge^0 U)$ be closed. Since $\cesobtloc(\wedge^0 U) = W^{1,n}_\loc(U; \R)$, there is a connected neighborhood $V \subset U$ of $x$ for which $u \vert_V \in W^{1, n}(V; \R)$. By \cite[Lemma 1.16]{HeinonenKilpelainenMartio2006book}, the restriction $u \vert_V$ is constant. Thus, a locally closed $0$-form in $\cesobtloc(\wedge^0 U)$ is locally constant, implying the exactness of \eqref{eq:stalkseq} at $(\cesheaf{0}{M})_x$.
\end{proof}

\subsection{Sheaf cohomology and the proof of Theorem \ref{derhamsobolevequivalence}}

Let $M$ be an oriented Riemannian manifold and $\Rsheaf_M$ the constant sheaf on $M$. The conformal sheaf cohomology $\cehom{*}(M;\Rsheaf_M)$ with coefficients in $\Rsheaf_M$ is the cohomology of the induced chain complex of vector spaces 
\begin{equation*}\label{eq:sheaf_cohomology}
	\xymatrix{
		0 \ar[r] 
		& \Gamma(M, \ceetale{0}{M}) \ar[r]^-{\Gamma(d)} 
		& \Gamma(M, \ceetale{1}{M}) \ar[r]^-{\Gamma(d)} 
		& \cdots,
	}
\end{equation*}
where the maps $\Gamma(d)$ are induced by the sheaf homomorphisms $d \colon \cesheaf{k}{M} \to \cesheaf{k+1}{M}$ and the linear isomorphisms $(\varphi_k)_M \colon \cesheaf{k}{M}(M) \to \Gamma(M, \ceetale{0}{M})$. We refer to Wells \cite[Theorem II.3.11]{Wells1980book} or Warner \cite[Sections 5.20-5.23]{Warner1983book} for more details and more general treatment of cohomologies $H^*(M;\sheaf)$ of $M$ having coefficients in a sheaf $\sheaf$.

For the proof of the naturality part of Theorem \ref{derhamsobolevequivalence}, we recall cohomomorphisms of sheaves and resolutions. For a more detailed treatment, we refer to Bredon \cite[Sections I.4, II.8]{Bredon1997book}.

Let $f\colon M\to N$ be a proper quasiregular mapping between oriented Riemannian $n$-manifolds $M$ and $N$. Then, for every $k \in \N$, the pull-back $f^* \colon \cesobtloc(\wedge^k N) \to \cesobtloc(\wedge^k M)$ of Sobolev forms induces a pull-back \emph{$f$-cohomomorphism} $f^* \colon \cesheaf{k}{N} \to \cesheaf{k}{M}$ of sheaves, that is, a collection of linear maps
\[
	\left\{ f^*_U \colon \cesobtloc(\wedge^k U) \to \cesobtloc(\wedge^k f^{-1}U) \colon U \subset N\ \text{open} \right\}
\]
satisfying 
\[
	f^*_U \circ i^*_{U, V} = i^*_{f^{-1}U, f^{-1}V} \circ f^*_V
\]
for all open $U \subset V \subset N$, where $i_{U, V}$ and $i_{f^{-1}U, f^{-1}V}$ are the inclusion maps $U \hookrightarrow V$ and $f^{-1}U \hookrightarrow f^{-1}V$, respectively. 
	
Let $\Rsheaf_M$ and $\Rsheaf_N$ denote the constant sheaves induced by $\R$ on $M$ and $N$ respectively. Since $f$ is continuous, it also induces a $f$-cohomomorphism $f^* \colon \Rsheaf_N \to \Rsheaf_M$ where every linear map $f^*_U \colon \Rsheaf_N(U) \to  \Rsheaf_M(f^{-1}U) $ is given by precomposition of functions.
\begin{lemma}\label{cohom_of_resolutions_lemma}
	The pull-back cohomomorphisms $f^* \colon \cesheaf{k}{N} \to \cesheaf{k}{M}$ form a $f$-cohomomorphism of resolutions extending $f^* \colon \Rsheaf_N \to \Rsheaf_M$, that is, the diagram
	\begin{equation}\label{eq:cohom_of_resolutions}
	\xymatrix{
		0 \ar[r] 
			&\Rsheaf_N \ar[r]^{i} \ar[d]^{f^*} 
			&\cesheaf{0}{N} \ar[r]^{d} \ar[d]^{f^*} 
			&\cesheaf{1}{N} \ar[r]^{d} \ar[d]^{f^*} 
			&\cdots\\
		0 \ar[r] 
			&\Rsheaf_M \ar[r]^{i} 
			&\cesheaf{0}{M} \ar[r]^{d} 
			&\cesheaf{1}{M} \ar[r]^{d} 
			&\cdots
	}
	\end{equation}
	commutes.
\end{lemma}
\begin{proof}
	Let $U \subset N$ be open. Since the map $f^* \colon \cesheaf{0}{N} \to \cesheaf{0}{M}$ is given by precomposition with $f$, the leftmost square of \eqref{eq:cohom_of_resolutions} commutes. The remaining squares commute due to Lemma \ref{qrchainmap}.
\end{proof}
	
We are now ready to recall the proof of the de Rham theorem in this context.

\begin{proof}[Proof of Theorem \ref{derhamsobolevequivalence}]
	For every $k \in \N$, let $\varphi_{k}$ be the natural presheaf isomorphism $\cesheaf{k}{M} \to \Gamma(\ceetale{k}{M})$. By definition of $\Gamma(d)$, the diagram 
	\begin{equation}\label{eq:cohom_eq_diagram}\begin{split}
		\xymatrixcolsep{3pc}\xymatrix{
			0 \ar[r]
				& \cesobtloc(\wedge^0 M) \ar[d]^-{(\varphi_{0})_M} \ar[r]^-{d}
				& \cesobtloc(\wedge^1 M) \ar[d]^-{(\varphi_{1})_M} \ar[r]^-{d}
				& \cdots\\
			0 \ar[r]
				& \Gamma(M, \ceetale{0}{M}) \ar[r]^-{\Gamma(d)}
				& \Gamma(M, \ceetale{1}{M}) \ar[r]^-{\Gamma(d)}
				& \cdots
		}
	\end{split}\end{equation}
	commutes. The upper complex in diagram \eqref{eq:cohom_eq_diagram} yields cohomology $\cehom{*}(M)$ and the lower complex the sheaf cohomology $\cehom{*}(M; \Rsheaf_M)$. Since $(\varphi_*)_M$ is a chain isomorphism, it induces a canonical isomorphism $\cehom{*}(M) \to \cehom{*}(M; \Rsheaf_M)$.
	
	The claim that $\cehom{*}(M)$ is canonically isomorphic to $H^*(M; \R)$ follows from the fact that all sheaf cohomologies derived from a fine resolution of $\Rsheaf_M$ by sheaves of vector spaces are canonically isomorphic, see Warner \cite[Sections 5.20-5.23]{Warner1983book} or Wells \cite[Theorem II.3.13 and Corollary II.3.14]{Wells1980book}. Indeed, the classical \emph{singular resolution} $\singsheaf{*}{M}$ with real coefficients is a fine resolution of $\Rsheaf_M$ and yields sheaf cohomology $H^*(M; \Rsheaf_M)$ canonically isomoprhic to $H^*(M; \R)$. We refer to Warner \cite[Sections 5.31-5.32]{Warner1983book} for details: note that the treatment is slightly more involved since the presheaf of singular cochains is not a sheaf. Now, the chain of canonical isomorphisms
	\[
		H^*(M; \R) \cong H^*(M; \Rsheaf_M) \cong \cehom{*}(M; \Rsheaf_M) \cong \cehom{*}(M)
	\]
	completes the first part of the proof.
	
	For the second part, let $f \colon M \to N$ be a proper and non-constant quasiregular map between oriented $n$-manifolds $M$ and $N$. By the natural isomorphisms $\cesheaf{*}{M} \to \Gamma(\ceetale{*}{M})$ and $\cesheaf{*}{N} \to \Gamma(\ceetale{*}{N})$, the pull-back cohomomorphisms $f^* \colon \cesheaf{k}{N} \to \cesheaf{k}{M}$ induce linear maps $\Gamma(f^*) \colon \Gamma(N, \ceetale{*}{N}) \to \Gamma(M, \ceetale{*}{M}) $. As consequence of Lemma \ref{cohom_of_resolutions_lemma}, $\Gamma(f^*)$ induces a linear map $f^* \colon \cehom{*}(N; \Rsheaf_N) \to \cehom{*}(M; \Rsheaf_M)$ which corresponds to the standard pull-back under the isomorphisms $\cehom{*}(M) \to \cehom{*}(M; \Rsheaf_M)$ and $\cehom{*}(N) \to \cehom{*}(N; \Rsheaf_N)$. 
	
	By continuity of $f$, similar pull-back cohomomorphisms $f^* \colon \singsheaf{*}{N} \to \singsheaf{*}{M}$ extending $f^* \colon \Rsheaf_N \to \Rsheaf_M$ are induced on the singular resolutions. This again induces a pull-back map $f^* \colon H^*(N; \Rsheaf_N) \to H^*(M; \Rsheaf_M)$ on singular sheaf cohomology which corresponds to the standard pull-back map $f^* \colon H^*(N; \R) \to H^*(M; \R)$. Finally, since the maps $f^* \colon H^*(N; \Rsheaf_N) \to H^*(M; \Rsheaf_M)$ and $f^* \colon \cehom{*}(N; \Rsheaf_N) \to \cehom{*}(M; \Rsheaf_M)$ both arise from a cohomomorphism of fine resolutions extending $f^* \colon \Rsheaf_N \to \Rsheaf_M$, they agree up to the canonical isomorphisms; see the discussion in Bredon \cite[Section II.8.1]{Bredon1997book} for details. This completes the proof.
\end{proof}


\section{Quasiregular push-forward}

In this section we discuss the push-forward operator $f_*$ on measurable differential forms induced by a quasiregular map $f \colon M\to N$ between closed, connected, oriented Riemannian manifolds. We refer to Heinonen--Kilpel\"ai\-nen--Martio \cite[pp. 263-268]{HeinonenKilpelainenMartio2006book} for the case of $0$-forms, i.e., measurable functions.

In order to define the quasiregular push-forward, we first recall a Vitali-type covering theorem on manifolds.
\begin{lemma}\label{vitalicoveringthm}
	Let $M$ be a closed Riemannian manifold, let $A \subset M$ be a measurable set, and let $r \colon A \to (0, \infty)$ be a function. Then there exists a countable disjoint collection $\mathcal{B}_A = \{ B_M(a_i, \rho_i) : i \in \N \}$ of open balls for which $A \setminus \cup \mathcal{B}_A$ has Lebesgue measure zero and $\rho_i \leq r(a_i)$ for each $i$.
\end{lemma}
\begin{proof}
	By \cite[Theorem 2.8.18.\ and Section 2.8.9.]{Federer1969book}, the claim holds for a collection of closed balls. Since the boundary of Riemannian balls has Lebesgue measure zero, the claim follows.	
\end{proof}

Our definition of $f_*$ is based on the following lemma.

\begin{lemma}\label{pushwelldefined}
	Let $k \in \{0, \ldots, n\}$ and let $f \colon M \to N$ be a non-constant $K$-quasiregular map between closed, connected, oriented Riemannian manifolds. Then there exist open sets $V_f \subset N$ and $U_{f,1}, \ldots, U_{f, \deg f} \subset M$ for which:
	\begin{enumerate}
		\item the sets $U_{f, i}$ are disjoint;
		\item the sets $V_f$ and $\bigcup_{i=1}^{\deg f} U_{f, i}$ have full measure in $N$ and $M$ respectively;
		\item for every $i \in \{1, \ldots, \deg f\}$, we have $f(U_{f,i}) = V_f$, and there exists a quasiconformal branch of the inverse $f^{-1}_i \colon V_f \to U_{f, i}$.
	\end{enumerate}
	Furthermore, if $k \in \{0, \ldots, n\}$ and $\omega \colon M \to \wedge^k T^*M$ is a measurable $k$-form on $M$, then for every $i \in \{1, \ldots, \deg f\}$ the pull-back $\left(f^{-1}_i\right)^* \omega$ defines a measurable $k$-form on $N$.
\end{lemma}

\begin{proof}
	Since $fM$ is a compact and open subset of $N$, $f$ is surjective. Let $B_f \subset M$ denote the \emph{branch set} of $f$, that is, the set where $f$ fails to be a local homeomorphism. The set $B_f$ is closed, and the sets $B_f$ and $fB_f$ have measure zero; see e.g.\ \cite[Proposition I.4.14]{Rickman1993book}. By compactness of $M$ and continuity of $f$, the image of the branch set $fB_f$ is also closed.
	
	Let $y \in N \setminus fB_f$. If $x \in f^{-1}\{y\}$, then $f$ is a local orientation-preserving homeomorphism at $x$, and consequently $i(x, f) = 1$. Since we have $\deg f = \sum_{x \in f^{-1} \{y\}} i(x, f)$, the set $f^{-1}\{y\}$ consists of exactly $\deg f$ different points. Since $f$ is a local homeomorphism, at each $x \in f^{-1}\{y\}$ we may fix a radius $r_y > 0$ for which $f^{-1} B_N(y, r_y)$ has exactly $\deg f$ connected components and $f$ is a homeomorphism on every such component. Note that this property also holds for any smaller radius $r < r_y$.
	
	We may now apply the Vitali covering theorem (Lemma \ref{vitalicoveringthm}) for the set $N \setminus fB_f$ and the function $y \mapsto r_y$, obtaining a collection $\mathcal{B} = \{B_1, B_2, \ldots\}$ of disjoint open balls. For each $B_j \in \mathcal{B}$, the pre-image $f^{-1} B_j$ has $\deg f$ components, which we denote by $U_{f, i, j}$ for $i \in \{1, \ldots, \deg f\}$. Let $V_f = \cup \mathcal{B}$ and $U_{f, i} = \bigcup_{B_j \in \mathcal{B}} U_{f, i, j}$ for $i \in \{1, \ldots, \deg f\}$. Since the sets $U_{f, i, j}$ are disjoint for a fixed $j$ and the sets $B_j = fU_{f, i, j}$ are disjoint, the sets $U_{f, i}$ are disjoint.
	
	Since $fB_f$ and $(N \setminus fB_f) \setminus \cup \mathcal{B}_f$ have measure zero, the set $N \setminus \cup \mathcal{B}_f$ has measure zero. Furthermore, since $f^{-1} V_f = \bigcup_{i=1}^{\deg f} U_{f, i}$, we have 
	\[
		M \setminus \bigcup_{i=1}^{\deg f} U_{f, i} = f^{-1} (N \setminus V_f),
	\]
	where $f^{-1} (N \setminus V_f)$ has measure zero due to the Lusin property of $f$. Hence, the sets $V_f$ and $\bigcup_{i=1}^{\deg f} U_{f, i}$ have full measure in $N$ and $M$, respectively.
	
	We note that, for each $B_j \in \mathcal{B}_f$ and $i \in \{1, \ldots, \deg f\}$, the restriction $f\vert_{U_{f, i, j}} \colon U_{f, i, j} \to B_j$ is a $K$-quasiregular homeomorphism, and hence $K$-quasiconformal. We denote by $f^{-1}_{i,j} \colon B_j \to U_{f, i, j}$ the inverse of the restriction $f\vert_{U_{f, i, j}}$. Then $f^{-1}_{i,j}$ is $K^{n-1}$-quasiconformal.
	
	We now define the maps $f^{-1}_i \colon V_f \to U_{f,i}$ by $f^{-1}_i\vert_{B_j} = f^{-1}_{i,j}$ for each $j$. Since the maps $f^{-1}_{i,j}$ are $K^{n-1}$-quasiregular and the sets $B_j$ are open and mutually disjoint, $f^{-1}_i$ is $K^{n-1}$-quasiregular for every $i \in \{1, \ldots, \deg f\}$. Furthermore, since the maps $f^{-1}_{i,j}$ are homeomorphisms and the image sets $U_{f, i, j} = f^{-1}_{i,j}(B_j)$ are mutually disjoint, the maps $f^{-1}_i$ are homeomorphisms, and hence quasiconformal.
	
	Finally, let $\omega \colon M \to \wedge^k T^*M$ be a measurable $k$-form on $M$, and fix $i \in \{1, \ldots, \deg f\}$. Let $\xi \colon N \to \wedge^k T^*N$ be a $k$-form defined by $\xi = (f^{-1}_i)^*\omega$ on $V_f$, and $\xi = 0$ on $N \setminus V_f$. 
	
	The form $\xi$ is measurable if and only if the coefficient functions $\xi_I$ in a local representation $\xi = \sum_I \xi_I dx_{i_1}\wedge \cdots \wedge dx_{i_k}$ are measurable, where $I=(i_1,\ldots, i_k)$. Let $W \subset \R$ be an open set. Since $B_j$ is open, $\xi \vert B_j = (f^{-1}_{i,j})^*\omega$, and the pull-back of a measurable form under a quasiregular map is measurable, we obtain that the set $\xi_I^{-1} W \cap B_j$ is measurable for every $B_j \in \mathcal{B}_f$. Furthermore, since the set $\xi_I^{-1} W \cap (N \setminus V_f)$ is contained in a set of measure zero, it is measurable. Hence, $\xi_I^{-1} W$ is measurable, completing the proof of measurability of $\xi_I$, and therefore of $\xi$.
\end{proof}

Due to the previous lemma, we obtain a well-defined quasiregular push-forward operator as follows.
	
\begin{defn}
	Let $k \in \{0, \ldots, n\}$ and $f \colon M \to N$ a non-constant quasiregular map between closed, connected, oriented Riemannian manifolds. Let $V_f$, $U_{f, i}$ and $f_i^{-1}$ be given by Lemma \ref{pushwelldefined}. For a measurable $k$-form $\omega \colon M \to \wedge^k T^*M$, we define a measurable $k$-form $\push{f}\omega \colon N \to \wedge^k T^*N$ by
	\[
		\push{f} \omega = \sum_{i=1}^{\deg f} \left(f^{-1}_i\right)^* \omega,
	\]
	where $\left(f^{-1}_i\right)^* \omega$ denotes the corresponding induced measurable form on $N$. 
\end{defn}

\begin{rem}
	We note that the resulting $\push{f} \omega$ is independent of the choice of $V_f$, $U_{f, i}$ and $f_i^{-1}$ up to measure zero, and therefore we may consider it without specifying $V_f$, $U_{f, i}$ and $f_i^{-1}$. Indeed, suppose $V_f'$, $U_{f, i}'$ and $(f_i^{-1})'$ also satisfy the conditions of Lemma \ref{pushwelldefined}. Then, for every $y \in V_f \cap V_f'$, we find a neighborhood $\bigcap_{x \in f^{-1}\{y\}} f(U_{f,x} \cap U_{f,x}')$ on which the branches of the inverse coincide, where $U_{f,x}$ and $U_{f,x}'$ denote the sets $U_{f, i}$ and $U_{f, i}'$ containing $x$, respectively. Hence, the two selections yield identical forms $\push{f} \omega$ on $V_f \cap V_f'$, which is a set of full measure in $N$.
\end{rem}

The main result of this section is the following theorem.
	
\begin{thm}\label{qrpushsobolev}
	Let $f \colon M \to N$ be a non-constant quasiregular map between closed, connected, oriented Riemannian manifolds, and let $\omega \in \cesobt(\wedge^k M)$. Then $\push{f} \omega \in \cesobt(\wedge^k N)$ with $d \push{f} \omega = \push{f} d\omega$.
\end{thm}

As an immediate corollary we obtain that $f_*$ is a chain map. Furthermore, we obtain that, in cohomology, $f_*$ is a left-inverse of the pull-back $f^*$ up to the multiplication by the degree of $f$.

\begin{cor}\label{qrpushcohom}
	Let $f \colon M \to N$ be a non-constant quasiregular map between closed, connected, oriented Riemannian manifolds. The linear map $\push{f} \colon \cesobt(\wedge^* M) \to \cesobt(\wedge^* N)$ is a chain map, and induces a linear map $\push{f} \colon \cehom{*}(M) \to \cehom{*}(N)$ satisfying
	\[
		\push{f} f^* [\omega] = (\deg f)[\omega].
	\]
	for each $[\omega] \in \cehom{*}(M)$.
\end{cor}
		
We prove Theorem \ref{qrpushsobolev} in a series of lemmas. The push-forward map $\push{f}$ is clearly linear. We begin by collecting some of the basic properties of $f_*$ in the following lemma.
	
\begin{lemma}\label{qrpushproperties}
	Let $f \colon M \to N$ and $g \colon N \to N'$ be non-constant quasiregular maps between closed, connected, oriented Riemannian $n$-manifolds, and let $k, l \in \{0, \ldots, n\}$ satisfy $k+l \leq n$. Then 
	\begin{enumerate}
		\item \label{pushprop:wedge} for all measurable forms $\alpha \colon M \to \wedge^k T^*M$ and $\beta \colon M \to \wedge^l T^*M$, 
		\[
			\push{f} (\alpha \wedge f^* \beta) = (\push{f} \alpha) \wedge \beta;
		\]
		\item \label{pushprop:inverse} for every measurable $k$-form $\omega \colon M \to \wedge^k T^*M$, 
		\[
			\push{f} f^* \omega = (\deg f) \omega;
		\]
		\item \label{pushprop:functor} for every measurable $k$-form $\omega \colon M \to \wedge^k T^*M$, 
		\[
			\push{(g \circ f)} \omega = \push{g} \push{f} \omega;
		\]
		\item \label{pushprop:integral} for every integrable $n$-form $\omega \in L^1(\wedge^n M)$,
		\[
			\int_N \push{f} \omega = \int_M \omega; 
		\]
	\end{enumerate}
\end{lemma}

\begin{proof}
	Fix $V_f$, $U_{f,i}$ and $f_i^{-1}$ according to Lemma \ref{pushwelldefined}. For \eqref{pushprop:wedge}, since $V_f$ is of full measure, it suffices to observe that
	\begin{align*}
		\push{f} (\alpha \wedge f^* \beta)\vert_{V_f}
		&= \sum_{i=1}^{\deg f} \left(f^{-1}_i\right)^* (\alpha \wedge f^* \beta)
		= \sum_{i=1}^{\deg f} \left(\left(f^{-1}_i\right)^*\alpha\right) \wedge \left(\left(f \circ f^{-1}_i\right)^*\beta\right)\\
		&= \sum_{i=1}^{\deg f} \left(\left(f^{-1}_i\right)^*\alpha\right) \wedge \beta
		= \left((\push{f} \alpha) \wedge \beta \right)\vert_{V_f}.
	\end{align*}
		
	Property \eqref{pushprop:inverse} is a corollary of \eqref{pushprop:wedge}. Indeed, let $\mathcal{X}_M$ be the characteristic function of $M$. Then, for every $y \in V_f$,
	\[
		(\push{f} \mathcal{X}_M)(y) = \sum_{i=1}^{\deg f} \mathcal{X}_M \circ f^{-1}_i(y) = \deg f.
	\]
	Hence, $\push{f} \mathcal{X}_M = (\deg f)\mathcal{X}_{N}$ almost everywhere on $N$, and therefore
	\[
		\push{f} f^* \omega = \push{f} \left( \mathcal{X}_M \wedge \left(f^* \omega\right)\right) 
		= \left(\push{f} \mathcal{X}_M\right) \wedge \omega = (\deg f) \mathcal{X}_{N} \wedge \omega = (\deg f) \omega.
	\]
	
	For \eqref{pushprop:functor}, we also fix $V_g$, $U_{g,j}$ and $g_j^{-1}$ as in Lemma \ref{pushwelldefined}, and define
	\begin{align*}
		V_{g \circ f} = V_g \cap g(V_f) &&
		U_{g \circ f, (i,j)} = U_{f,i} \cap f^{-1} U_{g,j}.
	\end{align*}
	It follows from the Lusin N property of $f$ and $g$ that $V_{g \circ f}$ and $\bigcup_{i,j} U_{g \circ f, (i,j)}$ have full measure. Furthermore, the rest of the conditions of Lemma \ref{pushwelldefined} also hold, with branches of the inverse $(g \circ f)_{(i,j)}^{-1} \colon V_{g \circ f} \to U_{g \circ f, (i,j)}$ given by 
	\[
		(g \circ f)_{(i,j)}^{-1} = f_i^{-1} \circ (g_j^{-1} \vert_{V_{g \circ f}}).
	\]
	Now, \eqref{pushprop:functor} follows by computing
	\begin{align*}
		&(\push{(g \circ f)}\omega)\vert_{V_{f \circ g}}
		= \sum_{j=1}^{\deg g} \sum_{i=1}^{\deg f}  (g_j^{-1} \vert_{V_{g \circ f}})^* (f_i^{-1})^* \omega\\
		&\qquad= \sum_{j=1}^{\deg g} (g_j^{-1} \vert_{V_{g \circ f}})^* (\push{f}\omega)
		= (\push{g} \push{f} \omega)\vert_{V_{f \circ g}}.
	\end{align*}
	
	Finally, \eqref{pushprop:integral} follows by the change-of-variables formula for the quasiconformal maps $f_i^{-1}$, since
	\begin{align*}
		\int_N \push{f} \omega &= \sum_{i=1}^{\deg f} \int_{V_f} \left(f^{-1}_i\right)^* \omega 
		= \sum_{i=1}^{\deg f} \int_{U_{f,i}} \omega = \int_M \omega.
	\end{align*}
\end{proof}

As the next step towards the proof of Theorem \ref{qrpushsobolev}, we show that the push-forward commutes with the (weak) exterior derivative. Towards this goal we state an auxiliary lemma.

\begin{lemma}\label{sobolevwedges}
	Let $f \colon M \to N$ be a non-constant quasiregular map between closed, connected, oriented Riemannian $n$-manifolds. Let $\omega \in \cesobt(\wedge^{k} M)$ and $\eta \in C^\infty(\wedge^{n-k-1}N)$, where $k \in \{0, \ldots, n-1\}$. Then $\omega \wedge f^*\eta$ has a weak differential in $L^1(\wedge^{n} M)$ satisfying
	\begin{equation}\label{eq:lemmawedgedifferential}
		d(\omega \wedge f^*\eta) = d\omega \wedge f^*\eta + (-1)^{k} \omega \wedge f^*d\eta.
	\end{equation}
\end{lemma}
	
\begin{proof}
	For $k > 0$ we have $\omega \in \cesob(\wedge^{k} M)$ and $f^*\eta \in \cesob(\wedge^{n-k-1} M)$, and the claim follows directly from Gol'dshtein--Troyanov \cite[Theorem 3.3]{GoldsteinTroyanov2010paper}. Thus it remains to prove the case $k = 0$. We follow here the strategy of the proof of \cite[Theorem 3.3]{GoldsteinTroyanov2010paper}.
		
	By the higher integrability of $f$, there exist $r > n/(n-1)$ and $s > 1$ for which $f^*\eta \in L^r(\wedge^{n-1} M)$ and $f^*d\eta \in L^{s}(\wedge^{n} M)$; see \cite[Chapter 2.3.3]{HajlaszIwaniecMalyOnninen2008paper} for the discussion in this case. Note that the $L^s$-integrability of $f^*d\eta$ also follows by Lemma \ref{qrpull-backhigherint} and the $L^r$-integrability of $f^*\eta$ is a corresponding result for $k$-forms proven analogously.
	
	Now, by Hölder's inequality,
	\[
		\norm{\omega \wedge f^*\eta}_\frac{n}{n-1} \leq \norm{\omega}_{\frac{nr}{nr - r - n}}\norm{f^*\eta}_{r}
	\]
	and
	\[
		\norm{d\omega \wedge f^*\eta + (-1)^{k} \omega \wedge df^*\eta}_1 
		\leq \norm{d\omega}_{n}\norm{f^*\eta}_{\frac{n}{n-1}} + \norm{\omega}_{\frac{s}{s - 1}}\norm{f^*d\eta}_{s}.
	\]
	Since $\omega \in \cesobt(\wedge^{0} M)$, $\omega \in L^p(\wedge^{0} M)$ for every $1 \leq p < \infty$, and consequently $\omega \wedge f^*\eta \in L^{n/(n-1)}(\wedge^{n-1} M)$ and $d\omega \wedge f^*\eta + (-1)^{k} \omega \wedge df^*\eta \in L^1(\wedge^n M)$.
			
	To show that $\omega \wedge f^*\eta$ has a weak differential satisfying \eqref{eq:lemmawedgedifferential}, let $t = \max\{s/(s-1), nr/(nr-r-n)\}$. We fix a sequence $(\tau_i)$ in $C^\infty(\wedge^{n-1} M)$ that converges to $f^*\eta$ in $W^{d, r, s}(\wedge^{n-1} M)$, and a sequence $(\omega_i)$ in $C^\infty(M)$ that converges to $\omega$ in $W^{d, t, n}(\wedge^0 M)$. Since $M$ has finite measure, we have, by H\"older's inequality, that 
	\begin{align*}
		\norm{\omega_i \wedge \tau_i - \omega \wedge f^*\eta}_{\frac{n}{n-1}}
		&\leq \norm{\omega_i}_\frac{nr}{nr-r-n}
			\norm{\tau_i - f^*\eta}_r + \norm{\omega_i - \omega}_\frac{nr}{nr-r-n}
			\norm{f^*\eta}_r\\
		&\leq C\left( \norm{\omega_i}_t
			\norm{\tau_i - f^*\eta}_r + \norm{\omega_i - \omega}_t\norm{f^*\eta}_r\right)
	\end{align*}
	and
	\begin{align*}
		&\norm{d(\omega_i \wedge \tau_i) 
			- d\omega \wedge f^*\eta + (-1)^{k} \omega \wedge df^*\eta}_1\\
		&\qquad\leq \norm{d\omega_i}_n \norm{\tau_i - f^*\eta}_\frac{n}{n-1} 
			+ \norm{d\omega_i - d\omega}_n \norm{f^*\eta}_\frac{n}{n-1}\\
		&\qquad\qquad\qquad +\norm{\omega_i}_\frac{s}{s-1} \norm{d\tau_i - f^*d\eta}_s 
			+ \norm{\omega_i - \omega}_\frac{s}{s-1} \norm{f^*d\eta}_s\\
		&\qquad\leq C\left(\norm{d\omega_i}_n \norm{\tau_i - f^*\eta}_r 
			+ \norm{d\omega_i - d\omega}_n \norm{f^*\eta}_r\right)\\
		&\qquad\qquad\qquad + C\left(\norm{\omega_i}_t \norm{d\tau_i - f^*d\eta}_s 
			+ \norm{\omega_i - \omega}_t \norm{f^*d\eta}_s\right),
	\end{align*}
	where the constant $C$ depends on $n$, $r$, $s$, and the volume of the manifold $M$.
	
	Now, the right hand sides of the previous two estimates tend to zero as $i$ tends to infinity. Hence, the sequence $(\omega_i \wedge \tau_i)$ converges to $\omega \wedge f^*\eta$ in $L^{n/(n-1)}$ and the sequence $(d(\omega_i \wedge \tau_i))$ converges to $d\omega \wedge f^*\eta + (-1)^{k} \omega \wedge f^*d\eta$ in $L^1$. Since the forms $\omega_i \wedge \tau_i$ are smooth, they are in $W^{d, n/(n-1), 1}(\wedge^{n-1} M)$, and Lemma \ref{weakdbysmoothlimit} yields that $d(\omega \wedge f^*\eta) = d\omega \wedge f^*\eta + (-1)^{k} \omega \wedge f^*d\eta$. The claim follows.
\end{proof}

We are now ready to prove that the push-forward $f_*$ commutes with the (weak) exterior derivative.
	
\begin{lemma}\label{qrpushweakder}
	Let $f \colon M \to N$ be a non-constant quasiregular map between closed, connected, oriented Riemannian $n$-manifolds, and let $\omega \in \cesobt(\wedge^k M)$ for some $k \in \{0, \ldots, n-1\}$. Then the measurable $k$-form $\push{f} \omega$ has a weak derivative satisfying $d \push{f} \omega = \push{f} d\omega$.
\end{lemma}

\begin{proof}
	Let $\eta \in C^\infty(\wedge^{k+1} N)$. By Lemma \ref{qrpushproperties},
	\begin{align*}
		\int_N \left<\push{f} d\omega, \eta\right> \vol_N
		= \int_N \left(\push{f} d\omega\right) \wedge \hodge\eta
		= \int_N \push{f} \left( d\omega \wedge f^*\hodge\eta \right).
	\end{align*}
	Furthermore, by Lemma \ref{sobolevwedges}, 
	\[
		d\omega \wedge f^*\hodge\eta 
		= d(\omega \wedge f^*\hodge\eta) - (-1)^{k} \omega \wedge df^*\hodge\eta
	\]
	and the $n$-form $d(\omega \wedge f^*\hodge\eta)$ is integrable. Thus, by Lemma \ref{qrpushproperties}, 
	\[
		\int_N \push{f} d(\omega \wedge f^*\hodge\eta) 
		= \int_M d(\omega \wedge f^*\hodge\eta) 
		= 0.
	\]
	Since
	\[
		(-1)^{k+1} df^*\hodge\eta 
		= (-1)^{nk-k+k+1} f^*\hodge\hodge d \hodge\eta 
		= f^* \hodge d^* \eta,
	\]
	we obtain
	\begin{align*}
		\int_N \left<\push{f} d\omega, \eta\right> \vol_N
		&= \int_N \push{f} \left( d\omega \wedge f^*\hodge\eta \right)
		= \int_N \push{f} \left( \omega \wedge (-1)^{k+1}df^*\hodge\eta \right)\\
		&= \int_N \push{f} \left( \omega \wedge f^*\hodge d^* \eta \right)
		= \int_N (\push{f} \omega) \wedge \hodge d^* \eta\\
		&= \int_N \left<\push{f} \omega, d^* \eta \right> \vol_N.
	\end{align*}
	Thus $\push{f} d\omega$ is the weak differential $d \push{f} \omega$ of $\push{f} \omega$.
\end{proof}

We continue with an $L^p$-estimate for the push-forward in the conformal exponent. In the following lemma, the space $L^{n/k}(\wedge^k M)$ for $k=0$ is the space $L^\infty(M)$.
 	
\begin{lemma}\label{qrpushnorm}
	Let $f \colon M \to N$ be a non-constant $K$-quasiregular map between closed, connected, oriented Riemannian $n$-manifolds, and let $\omega \in L^{n/k}(\wedge^k M)$ for some $k \in \{0, \ldots, n\}$. Then $\push{f} \omega \in L^{n/k}(\wedge^k N)$, and there exists a constant $C = C(n, k, K) \geq 0$ for which
	\[
		\norm{\push{f} \omega}_\frac{n}{k} \leq C \left(\deg f\right)^\frac{n-k}{n} \norm{\omega}_\frac{n}{k}.
	\]
\end{lemma}

\begin{proof}
	The case $k=0$ follows trivially from the definition and the Lusin property of $f$, and we may assume that $k\in \{1,\ldots, n\}$.
	
	Fix $V_f$, $U_{f, i}$ and $f_i^{-1}$ as in Lemma \ref{pushwelldefined}. Recall that the maps $f_i^{-1} \colon V_f \to U_{f, i}$ are $K'$-quasiconformal, where $K' = K'(K, n)$. Let $C = C(n, k, K')$ be the constant of Lemma \ref{qrnormestimate}. Then, by the elementary inequality 
	\[
		(a_1 + \ldots + a_l)^p \leq l^{p-1}(a_1^p + \ldots + a_l^p)
	\]
	for $p \geq 1$ and non-negative $a_1, \ldots, a_l$, we obtain 
	\begin{align}\label{eq:pushnormestimate}\begin{split}
		\left( \int_N \abs{\push{f} \omega}^\frac{n}{k} \vol_N \right)^\frac{k}{n} 
		&= \left( \int_{V_f} \abs{\left(f_1^{-1}\right)^*\omega + \ldots
			+ \left(f_{\deg f}^{-1}\right)^*\omega}^\frac{n}{k} \vol_N \right)^\frac{k}{n}\\
		&\leq \left((\deg f)^{\frac{n}{k}-1} \sum_{i=1}^{\deg f} \int_{V_f} 
			\abs{\left(f_i^{-1}\right)^*\omega}^\frac{n}{k} \vol_N \right)^\frac{k}{n}\\
		&\leq \left((\deg f)^{\frac{n-k}{k}} \sum_{i=1}^{\deg f} C \int_{U_{f,i}} 
			\abs{\omega}^\frac{n}{k} \vol_M \right)^\frac{k}{n}\\
		&= C^\frac{k}{n} (\deg f)^{\frac{n-k}{n}} \left(\int_{M} 
			\abs{\omega}^\frac{n}{k} \vol_M \right)^\frac{k}{n}.
	\end{split}\end{align}
	This concludes the proof.
\end{proof}

Finally, we show that the push-forward operator preserves the sharp and flat spaces $L^{1,\sharp}(\wedge^n M)$ and $L^{\infty,\flat}(\wedge^0 M)$. We formulate this property as follows.
	
\begin{lemma}\label{qrpushflatsharp}
	Let $f \colon M \to N$ be a non-constant quasiregular map between closed, connected, oriented Riemannian $n$-manifolds. Then the push-forward operator $\push{f}$ on measurable forms maps $L^{1, \sharp}(\wedge^n M)$ into $L^{1, \sharp}(\wedge^n N)$ and $L^{\infty, \flat}(\wedge^0 M)$ into $L^{\infty, \flat}(\wedge^0 N)$.
\end{lemma}
	
\begin{proof}
	Again, fix $V_f$, $U_{f, i}$ and $f_i^{-1}$ as in Lemma \ref{pushwelldefined}, and let $\eps > 0$ and $r > 1$ be such that $J_f^r$ and $J_f^{-\eps}$ are integrable over $M$. Since $f_i^{-1}$ is a right inverse of $f$, the chain rule yields $J_{f_i^{-1}} = J_f^{-1} \circ f_i^{-1}$. Now, for every $i \in \{1, \ldots, \deg f\}$,
	\begin{align*}
		\int_{V_f} J_{f_i^{-1}}^{1+\eps} \vol_N 
		&= \int_{V_f} J_{f_i^{-1}}^{\eps} J_{f_i^{-1}} \vol_N
		= \int_{V_f} \left(J_{f}^{-\eps} \circ f_i^{-1}\right) J_{f_i^{-1}} \vol_N\\
		&= \int_{U_{f,i}} J_{f}^{-\eps} \vol_M < \infty.
	\end{align*}
	Similarly, 
	\begin{align*}
		\int_{V_f} J_{f_i^{-1}}^{1-r} \vol_N 
		&= \int_{V_f} \left(J_{f}^{r-1} \circ f_i^{-1}\right) \vol_N
		= \int_{U_{f,i}} \left(J_{f}^{r-1} \circ f_i^{-1} \circ f\right) J_{f} \vol_M\\
		&= \int_{U_{f,i}} J_{f}^{r} \vol_M < \infty
	\end{align*}
	for every $i \in \{1, \ldots, \deg f\}$.
			
	Let $\omega \in L^{1, \sharp}(\wedge^n M)$. There exists $p > 1$ for which $\omega \in L^{p}(\wedge^n M)$. For every $i \in \{1, \ldots, \deg f\}$, since $J_{f_i^{-1}}^{1+\eps}$ is integrable, there exists $s > 1$ satisfying $\norm{(f_i^{-1})^* \omega}_{s} < \infty$; see \eqref{eq:higherpullbackn} for the computation. Similarly as in \eqref{eq:pushnormestimate}, we now obtain the estimate
	\[
		\norm{\push{f} \omega}_{s} 
		\leq \left((\deg f)^{s-1} \sum_{i=1}^{\deg f} \int_{V_f} 
		\abs{\left(f_i^{-1}\right)^*\omega}^{s} \vol_N \right)^\frac{1}{s} < \infty.
	\]
	Hence $\push{f}\omega \in L^{1, \sharp}(\wedge^n N)$.
			
	Similarly, let $u \in L^{\infty, \flat}(\wedge^0 M)$, and let $1 \leq p < \infty$. Since $u \in L^s(M)$ for every $1 \leq s < \infty$ and $J_{f_i^{-1}}^{1-r}$ is integrable, we have $\norm{(f_i^{-1})^* u}_{p} < \infty$ for every $i \in \{1, \ldots, \deg f\}$; see \eqref{eq:higherpullbackzero} for the computation. Thus $\norm{\push{f} u}_{p} < \infty$ and $\push{f}u \in L^{\infty, \flat}(\wedge^0 N)$.	
\end{proof}
	
Theorem \ref{qrpushsobolev} now follows immediately from Lemmas \ref{qrpushweakder}, \ref{qrpushnorm}, and \ref{qrpushflatsharp}.


\section{Norm in conformal cohomology}\label{sect:cohomology_norm}
	
In this chapter, we define norms in the conformal cohomology spaces $\cehom{*}(M)$ of a closed manifold $M$. We use a standard quotient norm construction; see e.g.\ Iwaniec--Scott--Stroffolini \cite[Section 7.1]{IwaniecScottStroffolini1999paper} and Bonk--Heinonen \cite[Section 3]{Bonk-Heinonen_Acta}. Furthermore, we obtain a norm estimate for the pull-back map $f^* \colon \cehom{*}(N) \to \cehom{*}(M)$ induced by a quasiregular map $f \colon M \to N$ between closed, connected, oriented manifolds. This estimate is a key part in the proof of Theorem \ref{alleigenvaluessamediag}.

Let $M$ be a closed Riemannian manifold, $k \in \{1, \ldots, n-1\}$, and let $\norm{\cdot}_{n/k} \colon \cehom{k}(M) \to [0, \infty)$ be the function
\[
	c \mapsto \inf_{\omega \in c} \norm{\omega}_{\frac{n}{k}}.
\]
The space $d\cesobt(\wedge^{k-1} M)$ is convex, and by Lemma \ref{confexpcompleteness}, it is a closed subspace of $L^{n/k}(\wedge^k M)$. Since each cohomology class $c\in \cehom{k}(M)$ is a closed affine subspace in the uniformly convex Banach space $L^{n/k}(\wedge^k M)$, there exists a unique $k$-form $\omega \in c$ for which $\norm{c}_{n/k} = \norm{\omega}_{n/k}$. Now, by a straightforward verification, $\norm{\cdot}_{n/k}$ is a norm on $\cehom{k}(M)$.

Our goal is to derive a version of Lemma \ref{qrnormestimate} in the cohomology norm for the pull-back $f^* \colon \cehom{*}(N) \to \cehom{*}(M)$. The upper bound follows directly from the upper bound in Lemma \ref{qrnormestimate}. To obtain the lower bound, we use Lemma \ref{qrpushnorm} together with the fact that $(\deg f)^{-1} \push{f}$ is a left inverse of the pull-back $f^*$.
	
\begin{thm}\label{cohomnormestimate}
	Let $f \colon M \rightarrow N$ be a non-constant $K$-quasiregular map between closed, connected, oriented Riemannian $n$-manifolds, and let $0 < k < n$. Then there is a constant $C = C(n, k, K) \geq 1$ for which
	\begin{equation}\label{eq:cohom_ineq}
		C^{-1}\left(\deg f\right)^{\frac{k}{n}} \norm{c}_\frac{n}{k}
		\leq \norm{f^*c}_\frac{n}{k}
		\leq C\left(\deg f\right)^{\frac{k}{n}} \norm{c}_\frac{n}{k}.
	\end{equation}
	for all $c \in \cehom{k}(N)$.
\end{thm}

\begin{rem}
	Note that for $k\in\{0,n\}$, the spaces $\cehom{k}(N)$ are one-dimensional and the mappings $f^* \colon \cehom{k}(N) \to \cehom{k}(M)$ are completely understood by Theorem \ref{derhamsobolevequivalence}. Indeed, given a continuous map $f\colon M \to N$ between closed, connected, oriented $n$-manifolds,  $f^* \colon H^*(N; \R) \to H^*(M; \R)$ maps the generator $[x\mapsto 1]$ of $H^0(N;\Z)$ to the generator $[x\mapsto 1]$ of $H^0(M;\Z)$ and the positive generator $c_N \in H^n(N;\Z)$ to $(\deg f)c_M\in H^n(M;\Z)$, where $c_M$ is the positive generator of $H^n(M;\Z)$.
\end{rem}

\begin{proof}[Proof of Theorem \ref{cohomnormestimate}]
	Let $c \in \cehom{k}(N)$ and let $\omega \in c$ be the $k$-form satisfying $\norm{\omega}_{n/k} = \norm{c}_{n/k}$. By Lemma \ref{qrnormestimate}, we have
	\begin{align*}
		\norm{f^*c}_\frac{n}{k}
		&\leq \norm{f^*\omega}_\frac{n}{k}
		\leq \left(C (\deg f)\int_M \abs{\omega}^\frac{n}{k} \vol_M\right)^\frac{k}{n}
		= C^{\frac{k}{n}} \left(\deg f\right)^\frac{k}{n} \norm{c}_\frac{n}{k},
	\end{align*}
	where the constant $C$ depends only on $n$, $k$, and $K$.
	
	To prove the other inequality, let $\tau \in f^*c$ be the $k$-form satisfying $\norm{\tau}_{n/k} = \norm{f^*c}_{n/k}$. 
	Then, by Corollary \ref{qrpushcohom},
	\[
		\push{f} \tau \in \push{f} f^* c = (\deg f) c.
	\]
	Hence, we obtain
	\begin{align*}
		\norm{f^*c}_\frac{n}{k}
		= \norm{\tau}_\frac{n}{k}
		\geq (C')^{-1} (\deg f)^{-\frac{n-k}{n}}\norm{\push{f} \tau}_\frac{n}{k}
		\geq (C')^{-1} (\deg f)^{\frac{k}{n}} \norm{c}_\frac{n}{k},
	\end{align*}
	where $C' = C'(n,k,K)$ is given by Lemma \ref{qrpushnorm}. This concludes the proof.
\end{proof}


\section{Eigenvalues and diagonalizability}
In this section, we prove Theorem \ref{alleigenvaluessamediag}. The result follows directly from Theorem \ref{cohomnormestimate} using some basic facts about complex vector spaces.

Recall that a linear map $L \colon V \to W$ between real vector spaces extends to a complex linear map $L \colon V \otimes \C \to W \otimes \C$ by the formula $L(v \otimes z) = L(v) \otimes z$ for $v \in V$ and $z \in \C$. We consider $V$ and $W$ as real subspaces of $V\otimes \C$ and $W\otimes \C$ for which we have $V + iV = V\otimes \C$ and $W+iW = W\otimes \C$. Under this identification $L \colon V \otimes \C \to W \otimes \C$ is given by the formula $L(v_1 + iv_2) = L(v_1) + iL(v_2)$ for $v_1, v_2 \in V$. Further, the complex eigenvalues of $L\colon V\to W$ correspond to eigenvalues of $L \colon V\otimes \C \to W\otimes \C$.

A norm $\norm{\cdot}$ in $V$ extends to a norm in the complex vector space $V \otimes \C$ by setting
\begin{equation}\label{eq:complex_norm_extension}
	\norm{v + iv'} = \sup_{\theta \in [0, 2\pi]} \norm{\cos(\theta)v + \sin(\theta)v'}
\end{equation}
for $v, v' \in V$; for details, see e.g.\ \cite{MunosSarantopuolosTonge1999paper}. Note in particular that the extended norm satisfies $\norm{zw} = \abs{z}\norm{w}$ for all $z \in \C, w \in V \otimes \C$.

The extension of the norm $\norm{\cdot}_{n/k}$ in $\cehom{k}(N)$ to $\cehom{k}(N) \otimes \C$ now yields a complex version of Theorem \ref{cohomnormestimate}.

\begin{lemma}\label{complexcohomnormestimate}
	Let $f \colon M \rightarrow N$ be a non-constant $K$-quasiregular map between closed, connected, oriented Riemannian  $n$-manifolds, and let $0 < k < n$. Then there is a constant $C = C(n, k, K) \geq 1$ for which
	\begin{equation*}
		C^{-1}\left(\deg f\right)^{\frac{k}{n}} \norm{c}_\frac{n}{k}
		\leq \norm{f^*c}_\frac{n}{k}
		\leq C\left(\deg f\right)^{\frac{k}{n}} \norm{c}_\frac{n}{k}.
	\end{equation*}
	for all $c \in \cehom{k}(N) \otimes \C$, where the complexification of the norm $\norm{\cdot}_{n/k}$ is as in \eqref{eq:complex_norm_extension}.
\end{lemma}

\begin{proof}
	Let $c = a+bi \in \cehom{k}(N) \otimes \C$, and let $C \geq 1$ be the constant in Theorem \ref{cohomnormestimate}. Then
	\begin{align*}
		\norm{f^*c}_\frac{n}{k}
		&= \sup_{\theta \in [0, 2\pi]} \norm{\cos(\theta)f^*a + \sin(\theta)f^*b}
		= \sup_{\theta \in [0, 2\pi]} \norm{f^*(\cos(\theta)a + \sin(\theta)b)}\\
		&\leq C\left(\deg f\right)^{\frac{k}{n}} \sup_{\theta \in [0, 2\pi]} 
			\norm{\cos(\theta)a + \sin(\theta)b}
		= C\left(\deg f\right)^{\frac{k}{n}} \norm{c}_\frac{n}{k},
	\end{align*}
	which yields the upper bound. The lower bound is obtained in the same manner.
\end{proof}

We now prove Theorem \ref{alleigenvaluessamediag} in two parts. We show first that for a $k \in \{1, \ldots, n-1\}$ and a uniformly quasiregular map $f \colon M \to M$ on a closed, connected, oriented Riemannian manifold $M$, each complex eigenvalue $\lambda$ of $f^* \colon \cehom{k}(M) \to \cehom{k}(M)$ has absolute value equal to $(\deg f)^\frac{k}{n}$.
	
\begin{thm}\label{eigenvaluessame}
	Let $f \colon M \to M$ be a non-constant uniformly $K$-quasiregular map on a closed, connected, oriented Riemannian $n$-manifold $M$, let $k \in \{1, \ldots, n-1\}$, and let $\lambda$ be a complex eigenvalue of $f^* \colon \cehom{k}(M) \to \cehom{k}(M)$. Then
	\[
		\abs{\lambda} = \left(\deg f\right)^{\frac{k}{n}}.
	\]
\end{thm}

\begin{proof}
	Let $c \in \cehom{k}(M)\otimes\C \setminus \{0\}$ be a complex eigenvector corresponding to the eigenvalue $\lambda$. Since every iterate $f^m$ of $f$ is $K$-quasiregular, we obtain by Lemma \ref{complexcohomnormestimate} the estimate
	\[
		C^{-1}\left(\deg f\right)^{\frac{mk}{n}} \norm{c}_\frac{n}{k}
		\leq \abs{\lambda}^m\norm{c}_\frac{n}{k}
		\leq C\left(\deg f\right)^{\frac{mk}{n}} \norm{c}_\frac{n}{k}
	\]
	for every $m \in \Z_+$, where $C= C(n, k, K)$ is independent of $m$. By rearranging the inequalities we obtain
	\[
		C^{-\frac{1}{m}} \leq \frac{\abs{\lambda}}{\left(\deg f\right)^{\frac{k}{n}}}  \leq C^{\frac{1}{m}},
	\]
	and, by letting $m \to \infty$, the claim follows.
\end{proof}
	
We prove the complex diagonalizability of $f^* \colon \cehom{k}(M) \to \cehom{k}(M)$ using the Jordan normal form of the matrix of $f^*$. Recall that, if $V$ is a finite-dimensional vector space and $L \colon V \to V$ is a linear map, then there exists a basis of $V\otimes\C$ under which the matrix representation of $L$ is zero outside of square blocks called \emph{Jordan blocks} on the diagonal, of the form
\[
	\begin{bmatrix}
		\lambda	&   1	&   0	&\ldots	&   0	\\
			0	&\lambda&   1	&\ldots	&	0	\\
			0	&	0	&\lambda&\ddots	&\vdots \\
		\vdots	&\vdots &\vdots &\ddots &	1	\\
			0	&	0	&	0	&	0	&\lambda	
	\end{bmatrix}
\]
where $\lambda$ is a complex eigenvalue of $L$; see e.g.\ \cite[Ch. 6]{Shilov1971book}. The Jordan normal form is unique up to the order of Jordan blocks. Clearly $L$ is diagonalizable if and only if it has a Jordan normal form consisting only of $1 \times 1$ blocks.
	
\begin{thm}
	Let $f \colon M \to M$ be a uniformly $K$-quasiregular map on a closed, connected, oriented Riemannian $n$-manifold $M$. Then $f^* \colon \cehom{k}(M) \to \cehom{k}(M)$ is complex diagonalizable.
\end{thm}

\begin{proof}
	Suppose that the Jordan normal form of $f^* \colon \cehom{k}(M) \to \cehom{k}(M)$ has a non-diagonal Jordan block associated to a complex eigenvalue $\lambda$. Then there exist $e_1, e_2 \in (\cehom{k}(M)\otimes\C) \setminus \{0\}$ satisfying
	\begin{align}\label{eq:jordaniterates}\begin{split}
		f^*e_1 &= \lambda e_1,\\
		f^*e_2 &= \lambda e_2 + e_1.
	\end{split}\end{align}
	Then, for each $m\in \Z_+$,
	\[
		(f^m)^*e_2 = \lambda^m\left(m\lambda^{-1}e_1 + e_2\right).
	\]
	Now, Lemma \ref{complexcohomnormestimate} yields
	\[
		\abs{\lambda}^m \norm{m\lambda^{-1}e_1 + e_2}_\frac{n}{k}
		= \norm{(f^m)^*e_2}_\frac{n}{k} 
		\leq C\left(\deg f\right)^{\frac{mk}{n}} \norm{e_2}_\frac{n}{k}
	\]
	for every $m\in \Z_+$, where $C = C(n, k, K)$ is independent of $m$. Since $\abs{\lambda} = (\deg f)^{k/n}$, we obtain
	\[
		m\norm{\lambda^{-1}e_1}_\frac{n}{k} \leq (C + 1)\norm{e_2}_\frac{n}{k},
	\]
	for every $m\in \Z_+$, which is a contradiction. Hence the Jordan normal form of $f^* \colon \cehom{k}(M) \to \cehom{k}(M)$ has no non-diagonal blocks.
\end{proof}
	
This completes the proof of Theorem \ref{alleigenvaluessamediag}.


\section{Degree restrictions}\label{sect:degree_limitation}	
In this section, we briefly elaborate on Corollary \ref{cor:degree_spectrum}. In its general form, the result is as follows.
 	
\begin{thm}\label{full_degree_corollary}
	Let $f\colon M\to M$ be a uniformly quasiregular self-map on a closed, connected, oriented Riemannian $n$-manifold $M$. Then for every $k \in \N$,
	\[
		(\deg f)^{\frac{k}{n} \dim H^k(M; \R)} \in \Z.
	\]
\end{thm}

We denote $d = \dim H^k(M; \R)$. The claim of Theorem \ref{full_degree_corollary} is nontrivial only for $0 < k < n$. Furthermore, as noted in the introduction, Theorem \ref{alleigenvaluessamediag} shows that $f^* \colon H^k(M; \R) \to H^k(M; \R)$ has a determinant equal to $\pm(\deg f)^{kd/n}$. Therefore, Theorem \ref{full_degree_corollary} follows immediately from the following lemma.

\begin{lemma}\label{topological_basis_lemma}
	Let $f \colon M \to M$ be a continuous self-map on a closed manifold, and let $k > 0$. Then there is a basis of $H^k(M; \R)$ under which the matrix of $f^* \colon H^k(M; \R) \to H^k(M; \R)$ has integer coefficients. In particular, the determinant of $f^* \colon H^k(M; \R) \to H^k(M; \R)$ is an integer.
\end{lemma}

The proof of Lemma \ref{topological_basis_lemma} is standard and straightforward for a reader familiar with algebraic topology. However, instead of searching for a reference, we give a simple proof.

\begin{proof}
	To avoid ambiguity, we denote by $f^*$ the pull-back $H^k(M; \R) \to H^k(M; \R)$ and by $f^*_\Z$ the pull-back $H^k(M; \Z) \to H^k(M; \Z)$. By a universal coefficient theorem for spaces of finite type, see \cite[Theorem 12.15]{Rotman1988book}, we obtain an isomorphism $\alpha \colon H^k(M;\Z)\otimes \R \to H^k(M;\R)$ which satisfies
	\begin{equation}\label{top_lemma_eq_1}
		\alpha \circ \left( f^*_{\Z}\otimes \id_\R \right) = f^* \circ \alpha.
	\end{equation}
	
	Let $T^k(M;\Z)$ be the torsion subgroup of $H^k(M;\Z)$ and let $H^k_\free(M;\Z)$ be the quotient group $H^k(M;\Z)/T^k(M;\Z)$ with projection $p \colon H^k(M;\Z)\to H^k_\free(M;\Z)$. We obtain an induced homomorphism $[f^*_\Z]\colon H^k_\free(M;\Z)\to H^k_\free(M;\Z)$ satisfying $p \circ f^*_\Z = [f^*_\Z]\circ p$. By the right exactness of the tensor product, the sequence
	\[
		\xymatrix{
			T^k(M, \Z) \otimes \R \ar[r]^-{i\otimes\id_\R} 
			& H^k(M, \Z) \otimes \R \ar[r]^-{p\otimes\id_\R} 
			& H^k_\free(M;\Z) \otimes \R \ar[r] 
			& 0
		}
	\]
	is exact, where $i$ is the inclusion homomorphism $T^k(M;\Z) \hookrightarrow H^k(M;\Z)$. Since $T^k(M, \Z) \otimes \R = 0$, we furthermore have that $p\otimes\id_\R$ is an isomorphism. We denote by $\beta$ the inverse of $p\otimes\id_\R$, and note that
	\begin{equation}\label{top_lemma_eq_2}
		\beta \circ \left( [f^*_\Z] \otimes \id_\R\right)  = \left( f^*_\Z \otimes \id_\R\right) \circ \beta.
	\end{equation}
 		
	Since $H^k_\free(M;\Z)$ is a finitely-generated free Abelian group, there exists $m \in \N$ and a free generating set $\{e_1, \ldots, e_m\}$ of $H^k_\free(M;\Z)$. Then $H^k_\free(M;\Z) \otimes \R$ is linearly isomorphic to $\R^m$, and $\{e_1 \otimes 1, \ldots, e_m \otimes 1\}$ is a basis of $H^k_\free(M;\Z) \otimes \R$. Let $\gamma \colon H^k_\free(M;\Z)\to H^k(M;\R)$ be the homomorphism $c \mapsto (\alpha \circ \beta) (c \otimes 1)$. Then, since $\alpha$ and $\beta$ are isomorphisms, $\{\gamma(e_1), \ldots, \gamma(e_m)\}$ is a basis of $H^k(M;\R)$. By \eqref{top_lemma_eq_1} and \eqref{top_lemma_eq_2}, we have $\gamma \circ [f^*_\Z] = f^* \circ \gamma$, and therefore the matrix of $f^*$ with respect to $\{\gamma(e_1), \ldots, \gamma(e_m)\}$ has integer coefficients. The claim that $\det f^*$ is an integer follows immediately.
\end{proof}


\end{document}